\documentclass[11pt,twoside]{amsart}
\usepackage{amsmath,amsfonts,amscd,amsthm}
\usepackage{amssymb}

\newcommand{\dbar}{\ensuremath{\overline\partial}}
\newcommand{\dbarstar}{\ensuremath{\overline\partial^*}}
\newcommand{\C}{\ensuremath{\mathbb{C}}}

\newcommand{\B}{\ensuremath{\mathbb{B}}}
\def\scriptf{{\mathcal F}}
\def\dom{\operatorname{dom}\,}

\def\kernel{\operatorname{\mathcal N}\,}
\def\range{\operatorname{\mathcal R}\,}

\makeatletter
\newcommand{\sumprime}{\if@display\sideset{}{'}\sum%
	\else\sum'\fi}
\makeatother

\newtheorem{thm}{Theorem}[section]
\newtheorem{prop}[thm]{Proposition}
\newtheorem{lem}[thm]{Lemma}
\newtheorem{cor}[thm]{Corollary}

\theoremstyle{definition}

\theoremstyle{remark}
\newtheorem{rem}[thm]{Remark}

\numberwithin{equation}{section}

\providecommand\ufootnote[1]{{\let\thefootnote\relax\footnote[0]{#1}}}

\newcommand{\cc}{\mathcal C}
\newcommand{\dc}{\mathcal D}
\newcommand{\ec}{\mathcal E}

\newcommand{\oc}{\mathcal O}

\newcommand{\nb}{\mathbb N}

\newcommand{\cb}{\mathbb C}
\newcommand{\zb}{\mathbb Z}

\newcommand{\ci}{{\mathcal C}^\infty}

\newcommand{\ol}{\overline}

\newcommand{\opa}{\ol\partial}
\newcommand{\wt}{\widetilde}

\DeclareMathOperator{\supp}{supp} \DeclareMathOperator{\im}{Im}
\DeclareMathOperator{\ke}{Ker}

 \DeclareMathOperator{\loc}{loc}

\begin{document}

\title{Hearing pseudoconvexity in Lipschitz domains with holes via $\bar{\partial}$}

\author{Siqi Fu, Christine Laurent-Thi\'{e}baut, and Mei-Chi Shaw}

\address{S. Fu, Department of Mathematical Sciences,
	Rutgers University, Camden, NJ 08102, USA} \email{sfu@rutgers.edu}

\address{C. Laurent-Thi\'{e}baut, Institut Fourier, Universit{\'e} Grenoble-Alpes, 38041 Grenoble, France}
\email{Christine.Laurent@univ-grenoble-alpes.fr}
\address{M.-C. Shaw, Department of Mathematics, University of Notre Dame, Notre Dame, IN 46556, USA}
\email{Mei-Chi.Shaw.1@nd.edu}
\date{}
\thanks{All three authors were partially supported by a grant from the  AGIR program of Grenoble INP  and Universit{\'e} Grenoble-Alpes, awarded to Christine Laurent-Thi\'{e}baut.  Siqi Fu and Mei-Chi Shaw were supported in part by NSF grants}

\begin{abstract} Let $\Omega=\widetilde{\Omega}\setminus \overline{D}$ where $\widetilde{\Omega}$ is a bounded domain with connected complement in $\C^n$ (or more generally in a Stein manifold) and $D$ is relatively compact open subset of $\widetilde{\Omega}$ with connected complement in $\wt{\Omega}$. We obtain characterizations of pseudoconvexity of $\widetilde{\Omega}$ and $D$ through the vanishing or Hausdorff property of the Dolbeault cohomology groups on various function spaces. In particular, we show that if the boundaries of $\widetilde{\Omega}$ and $D$ are Lipschitz and $C^2$-smooth respectively, then both $\widetilde{\Omega}$ and $D$ are pseudoconvex if and only if $0$ is not in the spectrum of the $\dbar$-Neumann Laplacian on $(0, q)$-forms for $1\le q\le n-2$ when $n\ge 3$; or $0$ is not a limit point of the spectrum of the $\dbar$-Neumannn Laplacian on $(0, 1)$-forms when $n=2$.
	
\bigskip
	
	
	\smallskip
	
\end{abstract}

\maketitle

\tableofcontents

\section{Introduction}\label{sec:intro}

The classical Theorem B of H. Cartan states that for any coherent analytic sheaf $\scriptf$ over a Stein manifold $X$, the sheaf cohomology groups $H^q(X, \scriptf)$ vanish for
all $q\ge 1$ (cf. \cite[Theorem~7.4.3]{Hormander90}).
The converse is also true (\cite[p.~65]{Serre53}; see \cite[pp.~86--89]{Hormander90} for a proof of this equivalence).  The $\opa$-complexes of the smooth forms, the  $L^2_{loc}$ forms, and the currents on an open set $\Omega$ in a complex manifold are all resolutions of the sheaf $\oc$ of germs of holomorphic functions. As a consequence, the Dolbeault cohomology groups obtained through these resolutions are all isomorphic to $H^q(\Omega,\oc)$ and will all be denoted by $H^{0, q}(\Omega)$. However, this isomorphism does not hold in general if one considers the Dolbeault cohomology groups for the $\opa$-complexes acting on function spaces with some growth or regularity properties on the boundary of the open set.

When $\Omega$ is a relatively compact pseudoconvex domain in a Stein manifold, it follows from H\"{o}r\-man\-der's $L^2$-existence theorem for the $\dbar$-operator that the $L^2$-Dolbeault cohomology groups $H^{0, q}_{L^2}(\Omega)$ vanish for $q\ge 1$. The converse of H\"{o}rmander's theorem also holds, under the assumption that the interior of the closure of $\Omega$ is the domain $\Omega$ itself.  Sheaf theoretic arguments for the Dolbeault cohomology groups (\cite{Serre53, Laufer66, Siu67, Broecker83, Ohsawa88}) can be modified to give a proof of this fact (see, e.g., \cite{Fu05, Fu10}).   We remark that some regularity of the boundary is necessary in order to characterize pseudoconvexity by the $L^2$-Dolbeault cohomology.
 The Dolbeault isomorphism also fails to hold between the usual Dolbeault cohomology groups and the $L^2$-Dolbeault cohomology groups.  For example, on the unit ball in $\C^2$ minus the center, the $L^2$-Dolbeault cohomology group $H^{0, 1}_{L^2}(\B^2\setminus\{0\})$ is trivial but the classical Dolbeault cohomology group $H^{0, 1}(\B^2\setminus\{0\})$ is infinite dimensional.

  It follows from  the work of Kohn \cite{Ko74} that, if $\Omega$ is a bounded pseudoconvex  domain in $\C^n$ with smooth boundary, then the Dolbeault cohomology groups $H^{0, q}_\infty(\overline\Omega)$ with   forms smooth up to the boundary
also vanish for all $q>0$ and the converse is also true.   But this does not hold   in general if the boundary of the domain is not smooth. For example, on the Hartogs' triangle ${\mathbb T}=\{(z, w)\in\C^2; \ |z|<|w|<1\}$, by \cite{Si}, we have that $H^{0,1}({\mathbb T})$ is trivial but $H^{0,1}_\infty(\ol{\mathbb T})$, the Dolbeault cohomology of forms with coefficients smooth up to the boundary,  is  infinite dimensional. In fact, it is not  even Hausdorff (see \cite{LS2015}).
In general, the Dolbeault cohomology groups with some growth or regularity properties  up to the boundary are not the same as the usual Dolbeault cohomology.

Andreotti and Grauert \cite{AnGr} showed that, for a relatively compact domain $\Omega$ with smooth boundary in a complex manifold that is strictly pseudoconvex (or more generally, satisfies Condition $a_q$\footnote{Recall that the boundary $b\Omega$ satisfies condition $a_q$ if the Levi form of a defining function has either at
least $q + 1$ negative eigenvalues or at least $n-q$ positive eigenvalues at every boundary point}), the Dolbeault cohomology group $H^{p,q}(\Omega)$ is finite dimensional. Furthermore,  by Grauert's bumping methods, the usual Dolbeault cohomology group is isomorphic to the Dolbeault cohomology group with regularity up to the boundary (see, e.g., \cite{HeLe2}).    It follows from H\"{o}rmander's work (\cite[\S~3.4]{Hormander65}) that on a bounded domain $\Omega$ with $C^3$-smooth boundary in a complex manifold, if the boundary $b\Omega$ satisfies the $a_q$ and $a_{q+1}$ conditions, then $H^{p, q}_{L^2}(\Omega)$ is isomorphic to $H^{p, q}(\Omega)$.  (See \cite{Ohsawa82} for relevant results.) In particular, when $\Omega$ is an annulus  between two $C^3$-smooth strictly pseudoconvex domains in a complex manifold,  the $L^2$ cohomology group  $H^{0, q}_{L^2}(\Omega)$ is finite dimensional when $0<q< n-1$.
When the domain $\Omega$ is an annulus between two weakly pseudoconvex domains with $C^3$ smooth boundary in $\C^n$, similar results  are proved in    \cite{Shaw85}.
The regularity of the boundary can be relaxed by assuming only that  $\Omega$ has $C^2$ boundary (see \cite{Shaw10}). In this case,   $H^{p, q}_{L^2}(\Omega)=0$  for all  $q\neq 0$
and $q\neq n-1$.

In this paper, we study the Dolbeault cohomolgy groups on various function spaces for a bounded domain $\Omega$ in $\C^n$ (or more generally in a Stein manifold) in the form of $\Omega=\widetilde{\Omega}\setminus \overline{D}$ where $\widetilde{\Omega}$ is a bounded domain with connected complement in $\C^n$ and $D$ is relatively compact open subset of $\widetilde{\Omega}$ with connected complement and with finitely many connected components. We obtain characterizations of pseudoconvexity of $\widetilde{\Omega}$ and $D$ through vanishing or Hausdorff property of the Dolbeault cohomology groups on various function spaces, including those of forms with coefficients smooth up to the boundary and of extendable currents.
In particular, we consider  $L^2$-Dolbeault cohomology groups and spectral theory for
$\opa$-Neumann Laplacian on the domain $\Omega$. We show that if  the outer boundary $b\widetilde{\Omega}$ is Lipschitz
 and the inner boundary $bD$ is $C^2$-smooth, then both $\widetilde{\Omega}$ and $D$ are pseudoconvex if and only if $0$ is not in the spectrum of the $\dbar$-Neumann Laplacian on $(0, q)$-forms for $1\le q\le n-2$ when $n\ge 3$; or $0$ is not a limit point for the spectrum of the $\dbar$-Neumannn Laplacian on $(0, 1)$-forms when $n=2$ (see Corollary \ref{eqL2},
 Theorems \ref{th:positivity1}, and  \ref{th:chaS} below).

 An earlier result in this spirit is due to Trapani in \cite{Tr3}. In his case, vanishing and Hausdorff properties of
the classical Dolbeault cohomology groups characterize the holomorphic convexity of $\ol D$; not just the pseudoconvexity of $D$.   Our results  and methods are different from the results in \cite{Tr3} (see the remark at the end of section 3).

The plan of the paper is as follows: In Section~2, we first recall the classical Serre duality and the definitions of various Dolbeault cohomology groups and their duals. In Section~3, we study the Dolbeault cohomology groups on forms that are smooth up to the boundary and on extendable currents. We obtain necessary and sufficient
conditions for the vanishing or Hausdorff properties of these Dolbeault cohomology groups.
Sections 4 and 5 are devoted to the $L^2$--theory for $\dbar$ on such a domain $\Omega$.  In Section 4, we study the relationship between the Dolbeault cohomology groups of $\Omega$ and those of  $\wt\Omega$ and $D$ on $L^2$ or $W^1$ spaces.  In Section 5, we establish necessary and sufficient conditions for such domains $\Omega$ that have vanishing $L^2$ or  $W^1$--Dolbeault cohomology groups. In particular, we show that when the boundary of $D$ is Lipschitz, the Dolbeault cohomology groups with $W^1_{loc}$--coefficients  are vanishing for $1\leq q<n-1$ and    Hausdorff for $q=n-1$  if and only if both $\wt\Omega$ and $D$  are peudoconvex (see Corollary \ref{eqW1}). When the inner boundary $bD$ is $\cc^2$ and the outer boundary $b\wt\Omega$ is Lipschitz, we obtain a characterization of pseudoconvexity of $\wt\Omega$ and $D$ by the vanishing and Hausdorff properties of the $L^2$--Dolbeault cohomology groups, which implies the above-mentioned statement that one can determine pseudoconvexity of $\wt\Omega$ and $D$ from spectral properties of the $\bar{\partial}$-Neumann Laplacian on $\Omega$.

\bigskip
\section{Dolbeault cohomology and the Serre duality}

Let $X$ be an $n$-dimensional complex manifold.
If $D\subset\subset X$ is a relatively compact subset of $X$, we  consider  the  $\opa$-complexes and their dual complexes  associated with  several spaces of forms  attached to $D$  (see section 2 in \cite{LaShdualiteL2} for more details)

We first recall definitions of several function spaces and their duals. Let $\ec(D)$ be the space  of $\ci$-smooth functions on $D$ with its classical Fr\'echet topology. It is well known that its dual can be identified  with the space $\ec'(X)$ of distributions with compact support in $D$. We will also consider the space $\ci(\ol D)$ of smooth functions on the closure of $D$; this is the space of the restrictions to $\ol D$ of $\ci$-smooth functions on $X$. It can also be identified  with the quotient of the space of $\ci$-smooth functions on $X$ by the ideal of functions that vanish with all their derivatives on $\ol D$. We endow $\ci(\ol D)$ with the Fr\'echet topology induced by $\ec(X)$. If $D$ has  Lipschitz boundary, the dual space of $\ci(\ol D)$ is the space $\ec'_{\ol D}(X)$ of distributions on $X$ with support contained in $\ol D$.  In general, when the boundary of $D$ is not necessarily Lipschitz, then  we only have that the   dual space of  $\ci(\ol D)$  is always a subspace of $\ec'_{\ol D}(X)$. We refer the reader to  \cite{LaShdualiteL2} for details.

Denote by $\dc_{\ol D}(X)$  the subspace of $\dc(X)$ consisting of  functions with support in $\ol D$, endowed with the natural Fr\'echet topology.
If $D$ has Lipschitz boundary, the dual of $\dc_{\ol D}(X)$ coincides with the space of restrictions to $D$ of distributions on $X$. This dual is called the space of extendable distribution on $D$ and will be denoted by $\check\dc'(\ol D)$. Moreover $\dc_{\ol D}(X)$ is a Montel space as a closed subspace of the Montel space $\ec(X)$ and hence reflexive, which implies that the dual space of $\check\dc'(\ol D)$, endowed with the strong dual topology, is the space $\dc_{\ol D}(X)$.  As before, when the boundary of $D$ is not Lipschitz,  we only have that the   dual space of $\check\dc'(\ol D)$   is a subspace of $\dc_{\ol D}(X)$.
\medskip

Recall that a \emph{cohomological complex of topological vector spaces} is a pair $(E^\bullet,\ d)$, where $E^\bullet=(E^q)_{q\in\zb}$ is a sequence of topological vector spaces and $d=(d^q )_{q\in\zb}$ is a sequence of densely defined closed linear maps $d^q$ from $E^q$ into $E^{q+1}$ that satisfy $d^{q+1}\circ d^q=0$. To any cohomological complex $(E^\bullet,\ d)$we associate cohomology groups $(H^q(E^\bullet))_{q\in\zb}$ defined by

$$H^q(E^\bullet)=\ker d^q/\im d^{q-1}$$
and endowed with the quotient  topology.
We fix $0\le p\le n$ and set  $E^q=0$ and $d^q\equiv 0$, if $q<0$ and $0\leq p\le n$.   We consider the $\bar{\partial}$-complex $(E^\bullet,\ d)$ with $d^q=\bar\partial$ if $0\le q\le n$ and $d^q\equiv 0$ if $q>n$ acting on:

\begin{enumerate}
	\item  $E^q=\ec^{p,q}(D)$, the space of $\ci$-smooth $(p,q)$-forms on $D$.
	
	\item   $E^q=\ci_{p,q}(\ol D)$, the space of $\ci$-smooth $(p,q)$-forms on $\ol D$.
		
	\item   $E^q=\check\dc'^{p,q}(\ol D)$, the space of extendable $(p,q)$-currents on $D$.
\end{enumerate}

The associated cohomology groups with  (1)-(3)   are denoted respectively by $H^{p,q}(D)$, $H^{p,q}_\infty(\ol D)$,   and $\check{H}^{p,q}(\ol D)$.

The dual complex of a cohomological complex $(E_\bullet,d)$ of topological vector spaces is the homological complex $(E'_\bullet,d')$,  where $E_\bullet=(E'_q)_{q\in\zb}$ with $E'_q$ the strong dual of $E^q$ and $d'=(d'_q )_{q\in\zb}$ with $d'_q$ the transpose of the map $d^q$.

Set $E'_q=0$ and $d'_q\equiv 0$ if $q<0$. If the domain $D$ has Lipschitz boundary, then the dual complexes of the previous cohomological complexes are $(E'_\bullet,d')$ with $d'_q=\opa$ for $0\leq q\leq n$ and:
\begin{enumerate}
	\item  $E'_q=\ec'^{n-p,n-q}(D)$, the space of currents with compact support in $D$.
	
	\item  $E'_q=\ec'^{n-p,n-q}_{\ol D}(X)$, the space of currents with compact support in $X$ whose support is contained in $\ol D$.

	\item   $E'_q=\dc_{\ol D}^{n-p,n-q}(X)$, the space of $\ci$-smooth forms with compact support in $\ol D$.
\end{enumerate}
 The associated cohomology groups are denoted respectively by $H^{n-p,n-q}_c(D)$, $H^{n-p, n-q}_{\ol D,cur}(X)$    and $H^{n-p, n-q}_{\ol D,\infty}(X)$.

The next proposition is a direct consequence of the Hahn-Banach Theorem.
\begin{prop}\label{adh}
	Let $(E^\bullet,d)$  and $(E'_\bullet,d')$ be two dual complexes, then
	$$\ol{\im d^q}=\{g\in E^{q+1}~|~\langle g,f\rangle =0, \ \forall f\in \ke d'_q\}.$$
\end{prop}

Let us recall the main result of the Serre duality (see \cite{Ca},\cite{Se} and \cite{LaLedualite}):

\begin{thm}\label{serre}
	Let $(E^\bullet,d)$ and $(E'_\bullet,d')$ be two dual complexes. Assume $H_{q+1}(E'_\bullet)=0$, then either $H^{q+1}(E^\bullet)=0$ or $H^{q+1}(E^\bullet)$ is not Hausdorff.
	
	If $(E^\bullet,d)$ is a complex of Fr\'echet-Schwartz spaces or of dual of Fr\'echet-Schwartz spaces, then, for any $q\in\zb$, $H^{q+1}(E^\bullet)$ is Hausdorff if and only if $H_{q}(E'_\bullet)$ is Hausdorff.
\end{thm}

 As a consequence of the previous theorem and of the solvability of the Cauchy-Riemann equation with prescribed support { in the closure of a bounded domain} with connected complement in a Stein manifold of dimension $n\geq 2$, in bidegree $(p,1)$, $0\leq p\leq n$, we obtain
\begin{thm}\label{sep}
	Let $X$ be a Stein manifold of complex dimension $n\geq 2$ and $D\subset\subset X$ a relatively compact subset of $X$ such that $X\setminus D$ is connected. Then
	
	(i) Either $H^{0,n-1}(D)=0$ or  $H^{0,n-1}(D)$ is not Hausdorff;
	
	{\parindent=0 pt if moreover $D$ has Lipschitz boundary,}
	
	(ii) Either $H^{0,n-1}_\infty(\ol D)=0$ or  $H^{0,n-1}_\infty(\ol D)$ is not Hausdorff;

	(iii) Either $\check H^{0,n-1}(\ol D)=0$ or  $\check H^{0,n-1}(\ol D)$ is not Hausdorff.
\end{thm}

For a more general  result and the proof of this theorem,  see Theorem 3.2 in \cite{LaShdualiteL2}.

\section{Characterization of pseudoconvexity by extendable Dolbeault cohomology}
\subsection{Necessary condition for the outside boundary}\label{s3}

Let $X$ be a Stein manifold of complex dimension $n\geq 2$ and $\Omega$ a relatively compact domain in $X$. Let us denote by $D$ the union of all the relatively compact connected components of $X\setminus\ol\Omega$ and set $\wt\Omega=\Omega\cup\ol D$.
Note that $\wt\Omega\setminus D$ is a closed subset of $\wt\Omega$ and we use $H^{0,q}_\infty(\wt\Omega\setminus D)$ to denote the Dolbeault cohomology groups for forms smooth up to the boundary of $D$.

In his paper \cite{Tr3}, Trapani proved that if $H^{0,q}(\Omega)=0$ for all $1\leq q\leq n-2$ and $H^{0,n-1}(\Omega)$ is Hausdorff, then $\wt\Omega$ has to be pseudoconvex. We now prove that this result extends to cohomology with growth or regularity conditions up to the boundary.

\begin{thm}\label{smooth}
Let $X$, $\Omega$ and $\wt\Omega$ be as above.  Then

(i) for each $q\in\nb$ with $1\leq q\leq n-2$, $H^{0,q}_\infty(\wt\Omega\setminus D)=0$ implies $H^{0,q}(\wt\Omega)=0$;

(ii)   $H^{0,n-1}_\infty(\wt\Omega\setminus D)$ is Hausdorff implies $H^{0,n-1}(\wt\Omega)=0$
\end{thm}
\begin{proof}
Let us first prove (i).
Let $f\in \ci_{0,q}(\wt\Omega)$ be a $\opa$-closed form. Then its restriction  to $\wt\Omega\setminus D$ is  a $\opa$-closed smooth form on $\wt\Omega\setminus D$. Since $H^{0,q}_\infty(\wt\Omega\setminus D)=0$, there exists a form $g\in \ci_{0,q-1}(\wt\Omega\setminus D)$ such that $f=\opa g$ on $\Omega$. Let $\wt g$ be a smooth extension of $g$ to $\wt\Omega$. Then $f-\opa\wt g$ is smooth and $\opa$-closed on $\wt\Omega$ and $f-\opa\wt g$ vanishes on $\wt\Omega\setminus\ol D$. Therefore we can extend $f-\opa\wt g$ by $0$ to a $\opa$-closed form on $X$. Since $X$ is Stein, there exists a smooth form $u$ on $X$ such that $f=\opa(\wt g+u)$ on $\wt\Omega$.

Let us now consider the assertion (ii). We first prove that $H^{0,n-1}(\wt\Omega)$ is Hausdorff.
Let $f\in \ci_{0,n-1}(\wt\Omega)$ be such that, for any $\opa$-closed $(n,1)$-current $T$ with compact support in $\wt\Omega$, $\langle T,f\rangle =0$. Then for any $(n,1)$-current $S$ with compact support in $\wt\Omega\setminus D$, we have $$\langle S,f\rangle =0.$$
In particular, $f$ is orthogonal to any current with coefficients in the dual space of $\cc^\infty(\wt\Omega\setminus D)$.

Using the assumptions that    $H^{0,n-1}_\infty(\wt\Omega\setminus D)$ is Hausdorff, we get  from Proposition \ref{adh}  the existence of a form $g\in \ci_{0,n-2}(\wt\Omega\setminus D)$ such that $f=\opa g$ on $\Omega$. Then following the proof of assertion (i) we obtain that $H^{0,n-1}(\wt\Omega)$ is Hausdorff. Since $X\setminus\wt\Omega$ is connected, it follows from Theorem \ref{sep} (see also Theorem 3.2 in \cite{LaShdualiteL2} and
Theorem 2 in \cite{Tr1}) that $H^{0,n-1}(\wt\Omega)=0$.
\end{proof}
It is well known (see, for example, Corollary 4.2.6 and Theorem 4.2.9 in \cite{Hormander90}) that a domain $U$ in $\cb^n$ is pseudoconvex if and only if we have $H^{0,q}(U)=0$ for all $1\leq q\leq n-1$.
   So from Theorem \ref{smooth} we can deduce:

\begin{cor}\label{pcsmooth}
Let $n\geq 2$ and $D\subset\subset\wt\Omega$ be two relatively compact open subsets of $\cb^n$ such that both $\cb^n\setminus\wt\Omega$ and $\wt\Omega\setminus D$ are connected.
Assume $H^{0,q}_\infty(\wt\Omega\setminus D)=0$, if  $1\leq q\leq n-2$, and $H^{0,n-1}_\infty(\wt\Omega\setminus D)$ is Hausdorff, then $\wt\Omega$ is pseudoconvex.
\end{cor}

Note that, replacing smooth forms up to the boundary by extendable currents in the proof of Theorem \ref{smooth}, we get:
\begin{prop}\label{pcext}
Let $n\geq 2$ and $D\subset\subset\wt\Omega$ be two relatively compact open subsets of $\cb^n$ such that both $\cb^n\setminus\wt\Omega$ and $\wt\Omega\setminus D$ are connected.
Assume that $\check{H}^{0,q}(\wt\Omega\setminus D)=0$, if  $1\leq q\leq n-2$, and $\check{H}^{0,n-1}(\wt\Omega\setminus D)$ is Hausdorff. Then $\wt\Omega$ is pseudoconvex.
\end{prop}

Let us notice that the vanishing or the Hausdorff property of the Dolbeault cohomology groups of the annulus $\Omega=\wt\Omega\setminus D$ are in fact independent of the larger domain $\wt\Omega$ as soon as it satisfies some cohomological conditions.

\begin{prop}\label{annulus}
Let  $D\subset\subset \wt\Omega_1\subset\subset\wt\Omega_2$ be bounded domains in $\cb^n$ such that $H^{0,q}(\wt\Omega_2)=0$ for some $q$,  $1\leq q\leq n-1$. Then we have the following:

(i) $H^{0,q}_\infty(\wt\Omega_1\setminus D)=0$ implies  $H^{0,q}_\infty(\wt\Omega_2\setminus D)=0$,

(i') $\check{H}^{0,q}(\wt\Omega_1\setminus D)=0$ implies  $\check{H}^{0,q}(\wt\Omega_2\setminus D)=0$,

and

(ii) $H^{0,q}_\infty(\wt\Omega_1\setminus D)$ is Hausdorff implies  $H^{0,q}_\infty(\wt\Omega_2\setminus D)$ is Hausdorff,

(ii') $\check{H}^{0,q}(\wt\Omega_1\setminus D)$ is Hausdorff implies  $\check{H}^{0,q}(\wt\Omega_2\setminus D)$ is Hausdorff.
\end{prop}

\begin{proof}
Let us first prove (i).
Let $f\in \ci_{0,q}(\wt\Omega_2\setminus  D)$ such that $\opa f=0$. Using the assumption $H^{0,q}_\infty(\wt\Omega_1\setminus D)=0$, we get a form $u\in \ci_{0,q-1}(\wt\Omega_1\setminus D)$, which satisfies $\opa u=f$ on $\wt\Omega_1\setminus\ol D$. Let $\chi$ be a smooth function on $\cb^n$ with support in $\wt\Omega_1$ and identically equal to $1$ on a neighborhood of $\ol D$. Consider $\opa(\chi u)$, it belongs to  $\ci_{0,q}(\wt\Omega_2\setminus  D)$ and is such that $f-\opa(\chi u)$ vanishes near the boundary of $D$. After extending by zero in $D$, we get a $\opa$-closed form on $\wt\Omega_2$, still denoted by $f-\opa(\chi u)$. Since $H^{0,q}(\wt\Omega_2)=0$, there exists $v\in \ci_{0,q-1}(\wt\Omega_2)$ such that $\opa v=f-\opa(\chi u)$. After restriction to $\wt\Omega_2\setminus D$, we get $f=\opa(\chi u+v)$ on $\wt\Omega_2\setminus\ol D$.

For assertion (ii), note that if $f\in \ci_{0,q}(\wt\Omega_2\setminus \ol D)$ is orthogonal to $\opa$-closed currents with compact support in  $\wt\Omega_2\setminus  D$, it is orthogonal to $\opa$-closed currents with compact support in   $\wt\Omega_1\setminus D$, then the proof follows the same argument used to prove (i).  The proofs of (i') and (ii')
are similar by substituting smooth forms with currents.
\end{proof}

\begin{cor}\label{invariance}
Let $D\subset\subset \wt\Omega_1\subset\subset\wt\Omega_2$ be bounded domains in $\cb^n$ such that $\wt\Omega_2$ is pseudoconvex.

(i) Assume that $H^{0,q}_\infty(\wt\Omega_1\setminus D)=0$, if  $1\leq q\leq n-2$, and $H^{0,n-1}_\infty(\wt\Omega_1\setminus D)$ is Hausdorff. Then $H^{0,q}_\infty(\wt\Omega_2\setminus D)=0$, if  $1\leq q\leq n-2$, and $H^{0,n-1}_\infty(\wt\Omega_2\setminus D)$ is Hausdorff.

(ii)  Assume that $\check{H}^{0,q}(\wt\Omega_1\setminus D)=0$, if  $1\leq q\leq n-2$, and $\check{H}^{0,n-1}(\wt\Omega_1\setminus D)$ is Hausdorff. Then $\check{H}^{0,q}(\wt\Omega_2\setminus D)=0$, if  $1\leq q\leq n-2$, and $\check{H}^{0,n-1}(\wt\Omega_2\setminus D)$ is Hausdorff.
\end{cor}

Note that by Corollary \ref{pcsmooth} and Proposition \ref{pcext},
 if $\Omega_1=\wt\Omega_1\setminus D$ is connected,
 each of the hypothesis in (i) or (ii) of Corollary \ref{invariance} forces $\wt\Omega_1$ to be pseudoconvex. Note also that if $\Omega_2=\wt\Omega_2\setminus D$ is connected,
  the condition $\wt\Omega_2$ pseudoconvex is necessary for the conclusion in (i) or (ii) of Corollary \ref{invariance} to hold.

\subsection{Necessary condition for the inside boundary}\label{s3bis}

Let $X$ be a connected, complex manifold and $D$ a relatively compact  open subset of  $X$ with Lipschitz boundary.  If $D$ is not a domain (i.e., a connected open set),  then the Lipschitz boundary assumption implies that  $D$ is a finite union of domains.  The  Dolbeault cohomology  groups  for $D$ are the direct sum of the corresponding cohomology groups on each connected components.  In this section we will  prove that  in any dimension $n\geq 2$ and for  some   Dolbeault cohomology groups on $X\setminus D$, such as forms  smooth up to the boundary, vanishing and Hausdorff properties of these groups  implies pseudoconvexity for the domain $D$, provided its boundary is sufficiently smooth.
A related result for $n=2$ was proved in \cite{Tr2}.

We will first relate the Hausdorff property or the vanishing of the cohomology groups with prescribed support in $\ol D$ with the same property for the cohomology groups of $X\setminus D$.
\begin{prop}\label{sep1}
Assume $H^{0,q-1}(X)=0$, for some $2\leq q\leq n$. We have

(i) if $H^{0,q}_{\ol D,\infty}(X)$ is Hausdorff, then $H^{0,q-1}_{\infty}(X\setminus D)$ is Hausdorff,

(ii) if $H^{0,q}_{\ol D,cur}(X)$ is Hausdorff, then $\check{H}^{0,q-1}(X\setminus D)$ is Hausdorff.
\end{prop}
\begin{proof}
Let $f\in\ci_{0,q-1}(X\setminus D)$ be a $\opa$-closed form such that for any $\opa$-closed $(n,n-q+1)$-current $T$ on $X$ with compact support in $X\setminus D$, we have $\langle T,\ f\rangle =0$. Let $\wt f$ be a smooth extension of $f$ to $X$, then $\opa\wt f$ is a $\opa$-closed smooth $(0,q)$-form on $X$ with support in $\ol D$. Let us prove that for any $\opa$-closed $(n,n-q)$-current $S$ on $D$ extendable to a current in $X$, we have $\langle S,\opa\wt f\rangle =0$. Let $\wt S$ be an extension of $S$ to $X$ with compact support, then, we get
$$\langle S, \opa\wt f\rangle =\langle\wt S,\opa\wt f\rangle =\langle\opa\wt S, \wt f\rangle $$
and since $T=\opa\wt S$ is a $\opa$-closed $(n,n-q+1)$-current on $X$ with compact support in $X\setminus D$, the orthogonality property of $f$ implies
$$\langle\opa\wt S, \wt f\rangle =\langle\opa\wt S, f\rangle =0.$$
By hypothesis $H^{0,q}_{\ol D,\infty}(X)$ is Hausdorff, therefore there is a smooth $(0,q-1)$-form $g$ in $X$ with support in $\ol D$ such that $\opa\wt f=\opa g$ on $X$. Hence $\wt f-g$ is a $\opa$-closed  smooth $(0,q-1)$-form on $X$, whose restriction to $X\setminus D$ is equal to $f$. Therefore, $H^{0,q-1}(X)=0$ implies that there exists a smooth $(0,q-2)$-form $h$ on $X$ such that $\opa h=\wt f-g$ and by restriction to $X\setminus D$ we get $\opa h=f$ on $X\setminus\ol D$. This proves that $H^{0,q-1}_{\infty}(X\setminus D)$ is Hausdorff. Assertion (ii) can be proved in the same way.
\end{proof}

\begin{prop}\label{sep2}
Assume $H^{0,q}(X)=0$  and either $H^{n,n-q+1}(X)=0$ or $H^{n,n-q+1}_c(X)=0$, for some $2\leq q\leq n$.
We have

(i) if $H^{0,q-1}_{\infty}(X\setminus D)$ is Hausdorff, then $H^{0,q}_{\ol D,\infty}(X)$ is Hausdorff,

(ii) if $\check{H}^{0,q-1}(X\setminus\ol D)$ is Hausdorff, then $H^{0,q}_{\ol D,cur}(X)$ is Hausdorff.
\end{prop}
\begin{proof}
Let $f$ be a $\opa$-closed $(0,q)$-form on $X$ with support contained in $\ol D$ such that for any  $\opa$-closed $(n,n-q)$-current $T$ on $D$ extendable as a current to $X$, we have $\langle T,f\rangle =0$. Since $H^{0,q}(X)=0$, there exists a smooth $(0,q-1)$-form $g$ on $X$ such that $\opa g=f$ on $X$. In particular, $\opa g=0$ on $X\setminus\ol D$.

Let $S$ be a $\opa$-closed $(n,n-q+1)$-current on $X$ with compact support in $X\setminus D$. Since $H^{n,n-q+1}(X)=0$ or $H^{n,n-q+1}_c(X)=0$, there exists a $(n,n-q)$-current $U$ on $X$ such that $\opa U=S$. (Here and hereafter we use $H^{n,n-q}_c(X)$ to denote the Dolbeault cohomology groups with compact support in $X$.) Hence $\opa U=0$ on $D$. Thus
$$\langle S,g\rangle =\langle  \opa U,g\rangle =\langle  U,\opa g\rangle =\langle U,f\rangle =0,$$
by hypothesis on $f$.
Therefore, the Hausdorff property of $H^{0,q-1}_{\infty}(X\setminus D)$ implies that there exists a smooth $(0,q-2)$-form $h$ on $X\setminus D$ such that $\opa h=g$. Let $\wt h$ be a smooth extension of $h$ to $X$. Then $u=g-\opa\wt h$ is a smooth form with support in $\ol D$ and
$$\opa u=\opa (g-\opa\wt h)=\opa g=f.$$
Assertion (ii) is proved in the same way.
\end{proof}

\begin{cor}\label{separe}
Let $X$ be a complex manifold of complex dimension $n\geq 2$ and $D$ be a relatively compact   open subset of  $X$ with Lipschitz boundary.
Assume $H^{0,q}(X)=0$ and $H^{0,q+1}(X)=0$, for some $1\leq q\leq n-1$, then
$$H^{0,q+1}_{\ol D,\infty}(X)~{\text{\sl is~ Hausdorff}}\quad\Leftrightarrow\quad H^{0,q}_\infty(X\setminus D)~{\text{\sl is~ Hausdorff}},$$
$$H^{0,q+1}_{\ol D,cur}(X)~{\text{\sl is~ Hausdorff}}\quad\Leftrightarrow\quad \check{H}^{0,q}(X\setminus\ol D)~{\text{\sl is~ Hausdorff}}.$$
\end{cor}
\begin{proof}
By Theorem \ref{serre}, the assumption $H^{0,q+1}(X)=0$ implies that $H^{n,n-q}_c(X)$ is Hausdorff.  Since $H^{0,q}(X)=0$, we have $H^{n,n-q}_c(X)=0$.
The corollary then follows by applying Propositions \ref{sep1} and \ref{sep2}.
\end{proof}

By applying the Serre duality for the complexes $({\ec'}^{n,\bullet}_{\ol D}(X),\opa)$ and $(\dc_{\ol D}^{n,\bullet}(X),\opa)$, we obtain:

\begin{cor}\label{sepdual}
Assume that $X$ is a Stein manifold  and $D$ a relatively compact  open subset of  $X$ with Lipschitz boundary. Then for any $1\leq q\leq n-1$,

(i) $H^{0,q}_{\infty}(X\setminus D)$ is Hausdorff if and only if $\check{H}^{n,n-q}(\ol D)$ is Hausdorff,

(ii) $\check{H}^{0,q}(X\setminus D)$ is Hausdorff if and only if $H^{n,n-q}_\infty(\ol D)$ is Hausdorff.
\end{cor}

In the special case $q=1$, the following theorem is a direct consequence of Corollary \ref{sepdual}  and Theorem \ref{sep}.

\begin{thm}\label{duality}
Let $X$ be a Stein manifold of complex dimension $n\geq 2$ and $D$ be a relatively compact  open subset of  $X$ with Lipschitz boundary such that $X\setminus D$ is connected. Then

(i) $H^{0,1}_\infty(X\setminus D)$ is Hausdorff if and only if  $\check{H}^{n,n-1}(\ol D)=0$,

(ii) $\check{H}^{0,1}(X\setminus D)$ is Hausdorff if and only if  $H^{n,n-1}_\infty(\ol D)=0$.
\end{thm}

\begin{thm}\label{vanish}
Let $X$ be a complex manifold of complex dimension $n\geq 2$ and $D$ be a relatively compact open subset of  $X$.
Assume $H^{0,q}(X)=0$, $H^{0,q+1}(X)=0$,
for some $1\leq q\leq n-1$, then
$$H^{0,q+1}_{\ol D,\infty}(X)=0\quad\Leftrightarrow\quad H^{0,q}_\infty(X\setminus D)=0;$$
$$H^{0,q+1}_{\ol D,cur}(X)=0\quad\Leftrightarrow\quad \check{H}^{0,q}(X\setminus D)=0.$$
\end{thm}
\begin{proof}
The proof is analogous to the proof of Corollary \ref{separe}, Propositions \ref{sep1}, and \ref{sep2}.
\end{proof}
\begin{cor}\label{dualvanish}
Let $X$ be a Stein manifold of complex dimension $n\geq 2$ and $D$ be a relatively compact  open subset of $X$ with Lipschitz boundary such that $X\setminus D$ is connected. Consider the assertions

(i) For all $1\leq q\leq n-2$, $\check{H}^{0,q}(X\setminus D)=0$ and $\check{H}^{0,n-1}(X\setminus D)$ is Hausdorff;

(ii) For all $2\leq q\leq n-1$, $H^{0,q}_{\ol D,cur}(X)=0$ and $H^{0,n}_{\ol D,cur}(X)$ is Hausdorff;

(iii) for all $1\leq q\leq n-1$, $H^{n,q}_\infty(\ol D)=0$.
\medskip

(i') For all $1\leq q\leq n-2$, $H^{0,q}_\infty(X\setminus D)=0$ and $H^{0,n-1}_\infty(X\setminus D)$ is Hausdorff;

(ii') For all $2\leq q\leq n-1$, $H^{0,q}_{\ol D,\infty}(X)=0$ and $H^{0,n}_{\ol D,\infty}(X)$ is Hausdorff;

(iii') for all $1\leq q\leq n-1$, $\check{H}^{n,q}(\ol D)=0$.
\medskip

{\parindent=0pt Then the triplets of assertions (i), (ii) and (iii), respectively (i'), (ii') and (iii'), are equivalent.}
\end{cor}
\begin{proof}
Using the  Serre duality the corollary follows from Corollary \ref{sepdual}, Theorem \ref{duality} and Theorem \ref{vanish}.
\end{proof}

Note that, under the hypotheses of Corollary \ref{dualvanish}, the assertions (ii) and (ii') still hold for $q=0$ and $q=1$ by analytic continuation and the fact that $H^{0,1}_c(X)=0$   and $X\setminus D$ is connected.
\medskip

As mentioned above in Section~\ref{s3}, a domain $D$ in $\cb^n$ is pseudoconvex if and only if we have $H^{0,q}(D)=0$ for all $1\leq q\leq n-1$.
Analogous results also hold for the Dolbeault cohomology of forms smooth up to the boundary for any $n$ and of extendable currents for $n=2$.

\begin{thm}\label{pcn}
Let $D\subset\cb^n$ be a domain such that interior($\ol D$)$=D$. If $H^{0,q}_\infty(\ol D)$  is finite dimensional for any $1\leq q\leq n-1$, then $D$ is pseudoconvex. Moreover, when $n=2$,   $D$ is pseudoconvex provided $\check{H}^{0,1}(\ol D)$ is finite dimensional.
\end{thm}
\begin{proof}
Following  Laufer's argument (\cite{Laufer75}; see also Theorem 5.1 \cite{Lalivre}),  we obtain that if $H^{0,q}_\infty(\ol D)$  (respectively $\check{H}^{0,q}(\ol D)$) is finite dimensional, then $H^{0,q}_\infty(\ol D)=0$  (respectively $\check{H}^{0,q}(\ol D)=0$). The Laufer's  argument  can be applied here because the spaces $\ec(\ol D)$ and $\check{\dc}'(\ol D)$ are invariant under differentiation and multiplication by polynomials. Thus we can assume that  $H^{0,q}_\infty(\ol D)=0$ for  all $1\leq q\leq n-1$ or  in the case when $n=2$, $\check{H}^{0,1}(\ol D)=0$.

The proof uses the forms introduced in \cite{Laufer66} (see also \cite{Fu05}). We will prove by contradiction.
Suppose that  $D$ is not pseudoconvex. Then there exists a domain $\wt D$ strictly containing $D$ such that any holomorphic function on $D$ extends holomorphically to $\wt D$. Since interior($\ol D$)$=D$, after a translation and a rotation we may assume that $0\in\wt D\setminus\ol D$ and there exists a point $z_0$ in the intersection of the plane $\{(z_1,\dots,z_n)\in\cb^2~|~z_1=\dots=z_{n-1}=0\}$ with $D$ that belongs to the same connected components (as the origin) of the  intersection of that plane with $\wt D$.

For any integer $q$ such that $1\leq q\leq n$ and $\{k_1,\dots,k_q\}\subset\{1, \ldots, n\}$, we set
$$
u(k_1,\dots,k_q)=\frac{(q-1)!}{|z|^{2q}}\sum_{j=1}^q (-1)^j\bar z_{k_j} \widetilde{d\bar
	z_{k_j}},
$$
where $\widetilde{d\bar
	z_{k_j}}=d\bar{z}_{k_1}\wedge
\ldots\wedge\widehat{d\bar{z}_{k_j}}\wedge\ldots\wedge
d\bar{z}_{k_q}$. (Here, as usual, $\widehat{d\bar{z}_{k_j}}$ indicates the deletion of
$d\bar z_{k_j}$ from the wedge product.)  Note that
$u(k_1, \ldots, k_q)$ is a smooth form on
$\C^n\setminus\{ 0\}$. Since $0\notin \overline{D}$, $u(k_1, \ldots, k_q)\in \ci_{(0, q-1)}(\ol{D})$.
Moreover, $u(k_1, \ldots, k_q)$ is skew-symmetric with respect to the
indexes in the tuple $(k_1, \ldots, k_q)$.  In particular, $u(k_1, \ldots, k_q)=0$ when two $k_j$'s are identical.
 A direct calculation yields that
\begin{equation}\label{dbar}
\opa u(k_1,\dots,k_q)=\sum_{l=1}^n z_l u(l, k_1,\dots,k_q).
\end{equation}
For any $1\leq q\leq n-1$, we consider the following assertion.

\smallskip\smallskip

\noindent{\bf H(q)}: For all integer $r<q$ and all multi-index $K=(k_1,\dots,k_r)$, setting $K=\emptyset$ if $r=0$, there exists a smooth $(0,n-r-2)$-form $v(K)$ on $\ol D$ such that
$$\opa v(K)=\sum_{j=1}^r (-1)^j z_{k_j} v(K\setminus k_j)+(-1)^{r+|K|} u((1,\dots,n)\setminus K),$$
where $|K|=k_1+\dots+k_r$ and
$(K \setminus J)$ denotes the tuple of remaining indexes after deleting those
in $J$ from $K$.

Note that $u(1,\dots,n)$ is a $\opa$-closed smooth $(0,n-1)$-form on $\cb^n\setminus\{0\}$ which contains $\ol D$. Since $H^{0,n-1}_\infty(\ol D)=0$, there exists a smooth $(0,n-2)$-form $v(\emptyset)$ on $\ol D$ such that $\opa v(\emptyset)= u(1,\dots,n)$. Therefore {\bf H(1)} is satisfied.

Let us prove now that if $1\leq q\leq n-2$ and {\bf H(q)} is satisfied,  then {\bf H(q+1)} is satisfied. It is sufficient to prove the existence of the $v(K)$'s satisfying the assertion {\bf H(q+1)} for any multi-index of length $q$, the other ones for $r<q$ already exist by {\bf H(q)}.
Let $K=(k_1,\dots,k_q)$, we set
$$w(K)=\sum_{j=1}^{q} (-1)^j z_j v(K\setminus k_ j)+(-1)^{q+|K|} u((1,\dots,n)\setminus K).$$
The $(0,n-q-1)$-form $w(K)$ is smooth on $\ol D$ and moreover, using \eqref{dbar} and the hypothesis {\bf H(q)}, a straightforward calculation proves that $\opa w(K)=0$ on $D$. Since $H^{0,n-q-1}_\infty(\ol D)=0$, there exists a smooth $(0,n-q-2)$-form $v(K)$ on $\ol D$ such that
$$\opa v(K)=w(K)=\sum_{j=1}^{q} (-1)^j z_j v(K\setminus k_ j)+(-1)^{q+|K|} u((1,\dots,n)\setminus K).$$

A finite induction process implies that {\bf H(n-1)} is satisfied and for $K=(1,\dots,n-1)$, we can consider the function
$$F=w(1,\dots,n-1)=\sum_{j=1}^{n-1} (-1)^j z_j v((1,\dots,n-1)\setminus j)-(-1)^{n+\frac{n(n-1)}{2}} u(n).$$
It is smooth on $\ol D$ and satisfies $\opa F=0$ on $D$. As $F$ is a  holomorphic function on $D$, it can be  extended holomorphically to $\wt D$. However, we have $F(0,\dots,z_n)=(-1)^{1+\frac{n(n-1)}{2}}\frac{1}{z_n}$ on $D\cap\{z_1=\dots=z_{n-1}=0\}$, which is holomorphic and singular at $z_n=0$, which gives the contradiction since $0\in\wt D\setminus D$.

When $n=2$, let us denote by $B(z_1,z_2)$ the$(0,1)$-form $\frac{\ol z_1~d\ol z_2-\ol z_2~d\ol z_1}{|z|^4}$ derived from the Bochner-Martinelli kernel in $\cb^2$, it is a $\opa$-closed form on $\cb^2\setminus\{0\}$. Then the $L^1$-function $\frac{\ol z_2}{|z|^2}$ defines a distribution in $\cb^2$ which satisfies $\opa(\frac{-\ol z_2}{|z|^2})=z_1B(z_1,z_2)$ on $\cb^2\setminus\{0\}$.
On the other hand, if  $\check{H}^{0,1}(\ol D)=0$, there exists an extendable distribution $v$ such that $\opa v=B$ on $D$ and by regularity of the $\opa$ in bidegree $(0,1)$, $v$ is smooth on $D$, since $B$ is smooth on $\cb^2\setminus\{0\}$. Set $F=z_1v+\frac{\ol z_2}{|z|^2}$, then $F$ is a holomorphic function on $D$, so it  extends holomorphically to $\wt D$, but we have $F(0,z_2)=\frac{1}{z_2}$ on $D\cap\{z_1=0\}$, which is holomorphic and singular at $z_2=0$. This  gives the contradiction since $0\in\wt D\setminus D$.
\end{proof}

Returning  to the case when  $\Omega=\wt\Omega\setminus\ol D$ is a bounded domain in $\cb^n$, where $D$ is the union of all relatively compact connected components of $\cb^n\setminus\ol\Omega$   as in section \ref{s3}.   We can easily derive the following corollary from Corollary \ref{dualvanish}
and the results of section \ref{s3}.

\begin{cor}\label{caract}
Let $D\subset\subset\wt\Omega$ be bounded open subsets of $\cb^n$, $n\geq 2$, such that $\cb^n\setminus\wt\Omega$ is connected. Assume $D$ has  Lipschitz boundary and $\Omega=\wt\Omega\setminus\ol D$ is connected. Consider the assertions

(i) For all $1\leq q\leq n-2$, $\check{H}^{0,q}(\wt\Omega\setminus D)=0$ and $\check{H}^{0,n-1}(\wt\Omega\setminus D)$ is Hausdorff;

(ii) for all $1\leq q\leq n-1$, $H^{n,q}_\infty(\ol D)=0$ and $\wt\Omega$ is pseudoconvex.
\medskip

(i') For all $1\leq q\leq n-2$, $H^{0,q}_\infty(\wt\Omega\setminus D)=0$ and $H^{0,n-1}_\infty(\wt\Omega\setminus D)$ is Hausdorff;

(ii') for all $1\leq q\leq n-1$, $\check{H}^{n,q}(\ol D)=0$ and $\wt\Omega$ is pseudoconvex.
\medskip

{\parindent=0pt Then the pairs of assertions (i) and (ii), respectively (i') and (ii'), are equivalent.}
\end{cor}

We are now in position to give characterizations of pseudoconvexity of the inside and outside boundaries for domains with holes in terms of their Dolbeault cohomolgy on various spaces,
depending on the regularity of the  boundary of the holes.

\begin{cor}\label{eqcur}
Let $D\subset\subset\wt\Omega$ be two relatively compact open subsets of $\cb^n$, $n\geq 2$, such that both $\cb^n\setminus\wt\Omega$ and $\wt\Omega\setminus D$ are connected. Assume $D$ has smooth boundary. Then $\check{H}^{0,q}(\wt\Omega\setminus D)=0$, for all $1\leq q\leq n-2$, and $\check{H}^{0,n-1}(\wt\Omega\setminus D)$ is Hausdorff if and only if $\wt\Omega$ and $D$ are pseudoconvex.
\end{cor}
\begin{proof}
The necessary condition is a direct consequence of  Corollary \ref{caract} and  Theorem \ref{pcn}. To get the sufficient condition, we use Kohn's result \cite{Ko74}, which asserts that if $D$ is a pseudoconvex domain with smooth boundary then $H^{n,q}_\infty(\ol D)=0$, for all $1\leq q\leq n-1$ and apply Corollary \ref{caract}.
\end{proof}

The following corollary is an analogous of Corollary~\ref{eqcur} for smooth forms in $\cb^2$ under the assumption that $D$ has Lipschitz boundary. The reason for the restriction to dimension~2 comes from the proof of Theorem~\ref{pcn}.  Note that in the cases of the extendable current cohomology, since the restriction of a distibution to a complex hyperplane does not exist in general (see also the remark before Lemma~\ref{lm:rest2}),  we need the regularity property for the $\opa$-operator which in this case holds only in bidegree $(0,1)$.

 \begin{cor}\label{eqsmooth}
Let $D\subset\subset\wt\Omega$ be two relatively compact open subsets of $\cb^2$ such that both $\cb^2\setminus\wt\Omega$ and $\wt\Omega\setminus D$ are connected. Assume $D$ has Lipschitz boundary. Then $H^{0,1}_\infty(\wt\Omega\setminus D)$ is Hausdorff if and only if $\wt\Omega$ and $D$ are pseudoconvex.
\end{cor}
\begin{proof}
The necessary condition is a direct consequence of Corollary \ref{caract} and  Theorem \ref{pcn}. The sufficient condition follows from Theorem 5 in \cite{ChaShdual}.
\end{proof}

\begin{rem} One can  only characterize  pseudoconvexity for such  domain  $\Omega$ with holes  by using  the right cohomology groups.  In an earlier paper  by  Trapani  \cite{Tr2},  he
proved that if $\Omega$ is the annulus between a pseudoconvex domain and   some
  Diederich-Fornaess worm  domain $D\Subset \C^2$   with smooth boundary,   the classical  Dolbeault cohomology group  $H^{0,1}(\Omega)$ is not
Hausdorff.  Thus if  we replace $H^{0,1}_\infty(\wt\Omega\setminus D)$ by $H^{0,1}(\wt\Omega\setminus \ol D)$,  Corollary  \ref{eqsmooth} does not hold.
\end{rem}

\section{Characterization of pseudoconvexity  by $L^2$ and $W^1$ Dolbeault cohomology}\label{sec:chaL2}

Let $X$ be a Stein  manifold of dimension $n\ge 2$, equipped with a hermitian metric. Let
$\Omega$ be a bounded domain in $X$. Let $L^2_{p,
	q}(\Omega)$ be the space of $(p, q)$-forms with $L^2$-coefficients.
Let $\dbar_{p, q}\colon L^2_{p, q}(\Omega)\to L^2_{p, q+1}(\Omega)$ be
the densely defined closed operator such that its domain consists of
all $f\in L^2_{p, q}(\Omega)$ such that $\dbar_{p, q}f$, defined in the
sense of distribution, is in $L^2_{p, q+1}(\Omega)$. Let $\dbarstar_{p, q}$ be its Hilbert space
adjoint. We drop the subscript  ${p.q}$ when there is no danger of confusion.

 Let
  $\dbar_c:L^2_{p,q}(\Omega)\to L^2_{p,q+1}(\Omega)$  be the minimal (strong) closure of $\dbar$. By this we mean that $f\in \text{Dom}(\dbar_c)$ if and only if that there exists a sequence of smooth forms  $f_\nu$ in $C^\infty_{p,q}(\Omega)$ compactly supported in $\Omega$ such that  $f_\nu\to f$ and $\dbar f_\nu\to \dbar f$ in $L^2$.
 Let $\vartheta=\dbar_c^*$ be the dual of $\dbar_c$. Then $\vartheta$ is equal
 to the maximal (weak)  $L^2$ closure of the operator  $\vartheta:L^2_{p,q}({\Omega})\to L^2_{p,q-1}({\Omega}).$
 We also define an operator $\dbar_{\tilde c}:L^2_{p,q}(\Omega)\to L^2_{p,q+1}(\Omega)$ to be the closure    of $\dbar$ such that $f$ is in the domain of $\dbar_{\tilde c}$ if and only
 if we extend $f$ to be zero outside $\overline \Omega$, there exists a $L^2_{p,q+1}(X)$ form $g$ supported in $\overline \Omega$ such that
 $\dbar f=g$ in $X$. The operator $\dbar_{\tilde c}$ corresponds to solving $\dbar$ with prescribed support on $\overline \Omega$ in the $L^2$ sense.

 We denote the   cohomology groups in $L^2_{p,q}(\Omega)$  with respect to $\dbar$ and   $\dbar_c$ by $H^{p,q}_{L^2}(\Omega)$ and   $H^{p,q}_{c,L^2}(\Omega)$. The cohomology
 group  for  $\dbar_{\tilde c}$  is denoted by $H^{p,q}_{\tilde{c}, L^2} (\Omega)$.  Note that
     $$H^{p,q}_{\tilde{c}, L^2} (\Omega)=H^{p,q}_{\overline\Omega,L^2}(X).$$

 Suppose that    $\Omega$ is a bounded   domain in $X$ with Lipschitz boundary.  Then  the weak and strong maximal (or minimal) $L^2$ extensions are the same by Friedrichs' lemma (see \cite{Hormander65}, \cite{ChenShaw01}).   When $\Omega$ has Lipschitz boundary, we have $\dbar_c=\dbar_{\tilde c}$ and
\begin{equation} \label{eq:weak}H^{p,q}_{c,L^2}(\Omega)=H^{p,q}_{\overline\Omega,L^2}(X).\end{equation}
(For a proof of this fact, see, e.g.,  Lemma 2.4 in \cite{LaShdualiteL2}).

The  cohomology groups for the $\bar{\partial}$-complexes on forms with   $W^1$   or $W^1_{\loc}$ coefficients are defined similarly and  denoted by
 $H^{p,q}_{W^1}(\Omega)$ and $H^{p,q}_{W^1_{loc}}(\Omega)$ respectively.   We will use $W^1_{loc}(X\setminus \Omega)$ to denote the space of functions with $W^1$ coefficients
 on compact subsets of $X\setminus \Omega$ and
  $H^{p,q}_{W^1_{loc}}(X\setminus \Omega)$ to denote  the Dolbeault cohomology groups for forms with coefficients in $W^1_{loc}(X\setminus \Omega)$.  Thus  the space  $H^{p,q}_{W^1_{loc}}(X\setminus \Omega)$ is different from the space of  $H^{p,q}_{W^1_{loc}}(X\setminus \ol \Omega)$, which is isomorphic to the usual Dolbeault cohomology group $H^{p,q}(X\setminus \ol \Omega)$.
 The dual space of $W^1_{loc}(X\setminus \Omega)$ is denoted by $W^{-1}(X\setminus \Omega)$. The dual complex  is denoted by $\dbar_c:  W^{-1}_{p,q-1}  (X\setminus \Omega)\to
 W^{-1}_{p,q} (X\setminus \Omega)$    and its corresponding cohomology groups are denoted by $H^{p,q}_{c,W^{-1}}(X\setminus \Omega)$.
Again, when the boundary of $\Omega$ is Lipschitz, this corresponds to solving $\dbar$ in $W^{-1}(X)$ spaces with compact support  in $X\setminus \Omega$.

We first establish the following   version of the Hartogs phenomenon.   Without loss of generality, we will deal with only $(0, q)$-forms.

\begin{lem}\label{lm:hartogs1} Let $\widetilde{\Omega}$ be a domain in a Stein manifold $X$ equipped with a hermitian metric and let $K$ be a compact subset of $\widetilde{\Omega}$.  Let $\Omega=\widetilde{\Omega}\setminus K$. Then
		  $\dbar_q^{\widetilde{\Omega}}$ has closed range provided $\dbar_q^{\Omega}$ has closed range.
		 \end{lem}

\begin{proof}   It suffices to show that there exists a constant $C>0$ such that for any $f\in\dom(\dbar^{\wt{\Omega}}_{q-1})$, one can find $u\in\dom(\dbar^{\wt{\Omega}}_{q-1})$ such that
	\begin{equation}\label{eq:cr1}
	\begin{cases}&\dbar u=\dbar f\qquad \text{on} \ \wt\Omega,
	\\& \|u\|_{\wt{\Omega}}\le C\|\dbar f\|_{\wt{\Omega}}.\end{cases}
	\end{equation}
	(see, e.g., \cite[Appendix 1]{Hormander04}). Evidently, the restriction $f\vert_\Omega$ of $f$ to $\Omega$ is in $\dom(\dbar^\Omega_{q-1})$. Thus under the assumption, there exists a $g\in\dom(\dbar^\Omega_{q-1})$ such that
	\[
	\begin{cases}&\dbar g=\dbar f \qquad  \text{ on } \Omega, \\& \|g\|_\Omega\le C_1\|\dbar f\|_\Omega,\end{cases}
	\]
	where  $C_1$ is  independent of $f$.  Let $\chi$ be a smooth cut-off function such that $0\le\chi\le 1$, $\chi=1$ in a neighborhood of $K$, and $\supp\chi\subset\subset\wt{\Omega}$. Consider $\alpha=\dbar f-\dbar (1-\chi)g=\chi\dbar f+\dbar\chi\wedge g$.  Then $\alpha$ is a $\dbar$-closed form in $L^2_{0, q}(\wt{\Omega})$ and $\supp\alpha\subset\subset\wt{\Omega}$. Extend $\alpha$ to $0$ outside $\wt{\Omega}$ and apply
	H\"{o}rmander's $L^2$-estimates to a bounded pseudoconvex domain $\widehat{\Omega}\supset\supset\wt{\Omega}$, we then obtain $v\in\dom(\dbar^{\widehat{\Omega}}_{q-1})$ such
	that
	\[
	\dbar v=\alpha \ \text{ on } \widehat\Omega, \qquad \|v\|_{\widehat\Omega}\le C_2\|\alpha\|_{\widehat\Omega}
	\]
	for some constant $C_2>0$.  Let $u=v+(1-\chi)g$.  Then $u$ satisfies the desired property~\eqref{eq:cr1}.

	  \end{proof}

\begin{lem}\label{lm:hartogs2}  Let $\wt{\Omega}$ be a bounded domain with Lipschitz boundary  in a Stein manifold $X$ of dimension $n$ such that $X\setminus\overline{\wt{\Omega}}$ is connected. Then $\range(\dbar^{*\wt\Omega}_{n-1})=\kernel(\dbar^{*\wt\Omega}_{n-2})$ where $\range$ and $\kernel$ denote   the range and the null spaces of the relevant operators.
\end{lem}

\begin{proof} Let $f\in\kernel(\dbar^{*\wt\Omega}_{n-2})$. Let $f^0$ be the extension of $f$ to $X$
	such that $f^0=0$ on $X\setminus\wt\Omega$.  Then $f^0$ is compactly supported and
	$\vartheta f^0=0$ in the sense of distribution on $X$ where $\vartheta$ denotes the formal adjoint of $\dbar$.  Let $*$ be the Hodge $*$-operator given by $\langle u,\ v\rangle dV=u\wedge *v$ where $dV$ is the volume form. Then $\vartheta=-*\bar{\partial}*$ (see, {\it e.g.}, \cite[\S 9.1]{ChenShaw01}). Thus $\dbar^{X}_{n, 1}(*f^0)=0$.

	Since $\wt\Omega$ has Lipschitz boundary,  we have from ( \ref{eq:weak}) that
$$H^{n, 1}_{c, L^2}(\wt\Omega)=H^{n,1}_{\ol{\wt\Omega},L^2}(X)=H^{n, 1}_{c, L^2}(X)=\{0\}.$$
Note that in the last two identities we use the assumptions that  
$X$ is Stein and $X\setminus\overline{\wt{\Omega}}$ is connected.
Thus there exists $u\in \dom(\dbar^{\wt\Omega}_c)$ such that $\dbar_c^{\wt\Omega} u=*f$ in $X$.
Since $*u$ is in the domain of $\dbar^{*\wt\Omega} $, we have 	  $\dbar^{*\wt\Omega} (*{u})=f$ on $\wt\Omega$. The lemma is proved.
\end{proof}

Let $\Omega$ be a relatively compact domain in $X$.
We denote by $D$ the union of all the relatively compact connected components of $X\setminus\ol\Omega$ and set $\wt\Omega=\Omega\cup\ol D$.

\begin{thm}\label{L2}
Let $X$, $\Omega$ and $\wt\Omega$ be as above.
Then

(i) for each   $1\leq q\leq n-2$, $H^{0,q}_{L^2}(\Omega)=0$   implies $H^{0,q}_{L^2}(\wt\Omega)=0$;

 (ii)  if   $\wt\Omega$ has Lipschitz boundary,  $H^{0,n-1}_{L^2}(\Omega)$    is Hausdorff implies $H^{0,n-1}_{L^2}(\wt\Omega)=0$.

 \smallskip
 \noindent
  Similarly, we have

 (iii) for each   $1\leq q\leq n-2$, $H^{0,q}_{W^1_{loc}}(\wt\Omega\setminus D)=0$ implies $H^{0,q}(\wt\Omega)=0$;

 (iv)  if $D$ has Lipschitz boundary,    $H^{0,n-1}_{W^1_{loc}}(\wt\Omega\setminus D)$ is Hausdorff implies $H^{0,n-1}(\wt\Omega)=0$.

\end{thm}
\begin{proof} The proof of (i) exactly the same  as in Theorem \ref{smooth}  or  Lemma \ref{lm:hartogs1}.

The proof of (ii) follows from  the $L^2$ Serre duality   (see \cite{LaShdualiteL2} and \cite{ChaShdual}).
We provide a simple proof here for the benefit of the reader.

 If $\range(\dbar^{\wt\Omega}_{n-2})$ is closed,
	then $H^{0, n-1}_{L^2}(\wt\Omega)$ is trivial.
  This is a direct consequence of Lemma \ref{lm:hartogs2}	 since
\begin{equation} \label{eq:4.2}
\range(\dbar^{\wt\Omega}_{n-2})=\kernel(\dbar^{*\wt\Omega}_{n-2})^\perp
=\range(\dbar^{*\wt\Omega}_{n-1})^\perp=\kernel(\dbar^{\wt\Omega}_{n-1}).	
\end{equation}
Note that in the first equality, we use the fact that $\range(\dbar^{\wt\Omega}_{n-2})$ is closed.

   The proof of (iii) is analogous to the proof of (i), using interior regularity we can choose $g$ with  $W^1$ coefficients on a neighborhood of the support of $\opa \chi$, if $f\in (W^1_{loc})_{0,q}(\wt\Omega)$. We get $H^{0,q}_{W^1_{loc}}(\wt\Omega)=0$ and by the Dolbeault isomorphism $H^{0,q}(\wt\Omega)=0$.
The proof of (iv) is similar to that of (ii) in Theorem \ref{smooth}.
\end{proof}

As in the smooth case, we have:

\begin{cor}\label{pcW1}
Let $n\geq 2$ and $D\subset\subset\wt\Omega$ be two relatively compact open subsets of $\cb^n$ such that both $\cb^n\setminus\wt\Omega$ and $\wt\Omega\setminus D$ are connected  and $D$ has Lipschitz boundary. Assume $H^{0,q}_{W^1_{loc}}(\wt\Omega\setminus D)=0$, if  $1\leq q\leq n-2$, and $H^{0,n-1}_{W^1_{loc}}(\wt\Omega\setminus D)$ is Hausdorff, then $\wt\Omega$ is pseudoconvex.
\end{cor}

From Theorem \ref{L2}, we deduce:

\begin{cor}\label{pcL2}
Let $n\geq 2$ and $D\subset\subset\wt\Omega$ be two relatively compact open subsets of $\cb^n$ such that both $\cb^n\setminus\wt\Omega$ and $\wt\Omega\setminus D$ are connected, and   $\wt\Omega$ has Lipschitz boundary.
Assume $H^{0,q}_{L^2}(\wt\Omega\setminus D)=0$, if  $1\leq q\leq n-2$, and $H^{0,n-1}_{L^2}(\wt\Omega\setminus D)$ is Hausdorff, then $\wt\Omega$ is pseudoconvex.
\end{cor}
\begin{proof}
 We will postpone the proof to the corollary in Theorem \ref{pcnL2}.
\end{proof}

As  in the smooth  case, the vanishing or the Hausdorff property of the Dolbeault cohomology groups of the annulus $\Omega=\wt\Omega\setminus D$ are in fact independent of the larger domain $\wt\Omega$ as soon as it satisfies some cohomological conditions.  The following was proved in \cite{ChLaSh}.

\begin{cor}\label{invarianceL2}
Let $D\subset\subset \wt\Omega_1\subset\subset\wt\Omega_2$ be bounded domains in $\cb^n$ such that $\wt\Omega_2$ is pseudoconvex.
Assume $H^{0,q}_{L^2}(\wt\Omega_1\setminus D)=0$, if  $1\leq q\leq n-2$, and $H^{0,n-1}_{L^2}(\wt\Omega_1\setminus D)$ is Hausdorff, then $H^{0,q}_{L^2}(\wt\Omega_2\setminus D)=0$, if  $1\leq q\leq n-2$, and $H^{0,n-1}_{L^2}(\wt\Omega_2\setminus D)$ is Hausdorff.
\end{cor}

Note that by  Corollary \ref{pcL2}, if $\Omega_1=\wt\Omega_1\setminus D$ is connected and $\Omega_1$ has Lipschitz boundary, the hypothesis of Corollary \ref{invarianceL2} forces $\wt\Omega_1$ to be pseudoconvex. Note also that if $\Omega_2=\wt\Omega_2\setminus D$ is connected and $\Omega_2$ has Lipschitz boundary, the condition $\wt\Omega_2$ pseudoconvex is necessary for the conclusion of Corollary \ref{invarianceL2} to hold.

  Next we will develop in the $L^2$ and $W^1$ settings what is done in section \ref{s3bis} for  forms smooth up to the boundary and extendable currents.

\begin{thm}\label{dualW1}
Let $X$ be a Stein manifold of complex dimension $n\geq 2$ and $D$ be a relatively compact  open subset of  $X$ with Lipschitz boundary such that $X\setminus D$ is connected. The following assertions are equivalent:

(i) For all $1\leq q\leq n-2$, $H^{0,q}_{W^1_{loc}}(X\setminus D)=0$ and $H^{0,n-1}_{W^1_{loc}}(X\setminus D)$ is Hausdorff;

(ii) For all $2\leq q\leq n-1$, $H^{n,q}_{c,W^{-1}}(X\setminus D)=0$ and $H^{n,n}_{c,W^{-1}}(X\setminus D)$ is Hausdorff;

(iii) For all $2\leq q\leq n-1$, $H^{0,q}_{\ol D,L^2_{loc}}(X)=0$, $H^{0,n}_{\ol D,L^2_{loc}}(X)$ is Hausdorff;

(iv) for all $1\leq q\leq n-1$, $H^{n,q}_{L^2}(D)=0$.
\end{thm}
\begin{proof}
The  Serre duality for the complexes $((W^1_{loc})^{0,\bullet}(X\setminus D),\opa)$, respectively $((L^2_{\ol D})^{n,\bullet}(X),\opa)$, implies the equivalence between (i) and (ii), respectively (iii) and (iv). Thus it is sufficient to prove that (i) implies (iii) and (iv) implies (ii).

Let us prove now that, for any $2\leq q\leq n$, if $H^{0,q-1}_{W^{1}_{loc}}(X\setminus D)$ is Hausdorff, then $H^{0,q}_{\ol D,L^2_{loc}}(X)$ is Hausdorff. Let $f$ be a $\opa$-closed $(0,q)$-form on $X$ with $L^2$ coefficients and support contained in $\ol D$ such that for any  $\opa$-closed $L^2$-form $u$ of bidegree $(n,n-q)$ on $D$, we have $\langle u,f\rangle =0$. Since $H^{0,q}(X)=0$ and by interior regularity, there exists a  form $g$ in $(W^1_{loc})_{0,q-1}(X)$ such that $\opa g=f$ on $X$, in particular $\opa g=0$ on $X\setminus\ol D$.

Let $S$ be a $\opa$-closed $(n,n-q+1)$-current on $X$ with compact support in $X\setminus D$ and coefficients in $W^{-1}(X)$, then, since $H^{n,n-q+1}_c(X)=0$, there exists an $(n,n-q)$-current $U$ with compact support on $X$ such that $\opa U=S$ and in particular $\opa U=0$ on $D$. Moreover $U$ can be chosen with $L^2_{loc}$ coefficients. Thus
$$\langle S,g\rangle =\langle \opa U,g\rangle =\langle U,\opa g\rangle =\langle U,f\rangle =0,$$
by hypothesis on $f$.

Therefore the Hausdorff property of $H^{0,q-1}_{W^1_{loc}}(X\setminus D)$ implies there exists a  $(0,q-2)$-form $h$ on $X\setminus D$ coefficients in $W^1_{loc}(X\setminus D)$ such that $\opa h=g$. Let $\wt h$ be a $W^1_{loc}$-extension of $h$ to $X$ then $u=g-\opa\wt h$ is a $L^2$-form with support in $\ol D$ and
$$\opa u=\opa (g-\opa\wt h)=\opa g=f.$$

In the same way we can prove that, for any $2\leq q\leq n-1$, if $H^{0,q-1}_{W^1_{loc}}(X\setminus D)=0$, then
$H^{0,q}_{\ol D,L^2_{loc}}(X)=0$.

Assume now that for some $1\leq q\leq n-2$, $H^{n,q}_{L^2}(D)=0$. Let $f$ be a $\opa$-closed form in $W^{-1}_{n,q+1}(X)$ with compact support in $X\setminus D$. Since $X$ is a Stein manifold, there exists a $(n,q)$-form $g$ with compact support in $X$ such that $\opa g=f$ and by interior regularity $g$ can be chosen with $L^2$ coefficients on $D$. Since the support of $f$ is contained in $X\setminus D$, $g$ is $\opa$-closed in $D$ and as $H^{n,q}_{L^2}( D)=0$, we get $g=\opa h$ for some $(n,q-1)$-form $h$ in $L^2_{n,q-1}(D)$. Let $\wt h$ be the extension of $h$ by $0$ outside of $D$. Then $g-\opa\wt h$ vanishes on $D$ and satisfies $\opa(g-\opa\wt h)=f$. This shows $H^{n,q+1}_{c,W^{-1}}(X\setminus D)=0$.

To end the proof, assume $H^{n,n-1}_{L^2}(D)=0$.  Let $f$ be a $\opa$-closed form in $W^{-1}_{n,n}(X)$ with compact support in $X\setminus D$ orthogonal to the $\opa$-closed functions, which are $W^1$ in $X\setminus D$ and in particular to the holomorphic functions in $X$. The Hausdorff property of $H_{c,W^{-1}}^{n,n}(X)$ implies that there exists an $(n,n-1)$-form $g$ with compact support in $X$ such that $\opa g=f$ and by interior regularity, $g$ can be chosen with $L^2$ coefficients on $D$. As for $1\leq q \leq n-2$, we can conclude that $f=\opa u$, where $u$ is in $W^{-1}_{n,n-1}(X)$ with compact support in $X\setminus D$.
\end{proof}

\begin{thm}\label{dualL2}
Let $X$ be a Stein manifold of complex dimension $n\geq 2$ and $D$ be a relatively compact   open subset of  $X$ with Lipschitz boundary such that $X\setminus D$ is connected. The following assertions are equivalent:

(i) For all $1\leq q\leq n-2$, $H^{0,q}_{L^2}(X\setminus D)=0$ and $H^{0,n-1}_{L^2}(X\setminus D)$ is Hausdorff

(ii) For all $2\leq q\leq n-1$, $H^{n,q}_{c,L^2}(X\setminus D)=0$ and $H^{n,n}_{c,L^2}(X\setminus D)$ is Hausdorff

(iii) for all $1\leq q\leq n-1$, $H^{n,q}_{W^1}( D)=0$.
\end{thm}
\begin{proof}
The equivalence between (i) and (ii) is a direct consequence of the Serre duality (see \cite{ChaShdual} or Theorem \ref{serre}).
From Proposition 4.7 in \cite{LaShdualiteL2}, we know that $H^{n,n}_{c,L^2}(X\setminus D)$ is Hausdorff if and only if $H^{n,n-1}_{W^1}( D)=0$.
Let us now prove the equivalence between (ii) and (iii) for the other degrees.

Assume (ii) is satisfied. Let $f\in W^1_{n,q}(D)$, $1\leq q\leq n-2$, be a $\opa$-closed form on $D$ and let $\wt f$ be a $W^1$ extension with compact support of $f$ to $X$. The $(0,q+1)$-form $\opa\wt f$ has $L^2$ coefficients and compact support in $X\setminus D$. By (ii), $H^{n,q+1}_{c,L^2}(X\setminus D)=0$ and therefore there exists a form $g\in L^2_{n,q}(X)$ with compact support in $X\setminus D$ such that $\opa\wt f=\opa g$. So the form $\wt f-g$ is a $\opa$-closed form on $X$ whose restriction to $D$ is equal to $f$. As $X$ is a Stein manifold, $\wt f-g=\opa h$ for some $L^2_{loc}$-form $h$ on $X$. Then we have $f=\opa h$ on $D$ and it follows from interior regularity that we can choose $h$ to belong to $W^1_{n,q-1}(D)$, which proves $H^{n,q}_{W^1}( D)=0$.

Let us prove the converse. Let $f\in L^2_{n,q}(X)$, $2\leq q\leq n-1$, be a $\opa$-closed form with compact support in $X\setminus D$. Since $X$ is a Stein manifold, there exists a $(n,q-1)$-form $g$ with compact support in $X$ such that $\opa g=f$ and by interior regularity $g$  can be chosen with  $W^1$ coefficients on $D$. Since the support of $f$ is contained in $X\setminus D$, $g$ is $\opa$-closed in $D$ and as $H^{n,q-1}_{W^1}( D)=0$, we get $g=\opa h$ for some $(n,q-2)$-form $h$ in $W^1_{n,q-2}(D)$. Let $\wt h$ be a $W^1$ extension of $h$ with compact support in $X$, which exists since $D$ is a relatively compact domain with Lipschitz boundary. Then $g-\opa\wt h$ vanishes on $D$ and satisfies $\opa(g-\opa\wt h)=f$. This shows $H^{n,q}_{c,L^2}(X\setminus D)=0$.
\end{proof}

\begin{cor}\label{caractL2}
Let $D\subset\subset\wt\Omega$ be bounded open subsets of $\cb^n$, $n\geq 2$, such that $\cb^n\setminus\wt\Omega$ is connected. Assume $D$ has  Lipschitz boundary and $\Omega=\wt\Omega\setminus\ol D$ is connected. Consider the assertions:

(i) For all $1\leq q\leq n-2$, $H^{0,q}_{W^1_{loc}}(\wt\Omega\setminus D)=0$ and $H^{0,n-1}_{W^1_{loc}}(\wt\Omega\setminus D)$ is Hausdorff;

(ii) For all $1\leq q\leq n-1$, $H^{n,q}_{L^2}(D)=0$ and $\wt\Omega$ is pseudoconvex.
\medskip

{\parindent=0pt and if $\wt\Omega$ has Lipschitz boundary}

(i') For all $1\leq q\leq n-2$, $H^{0,q}_{L^2}(\Omega)=0$ and $H^{0,n-1}_{L^2}(\Omega)$ is Hausdorff;

(ii') For all $1\leq q\leq n-1$, $H^{n,q}_{W^1}( D)=0$ and $\wt\Omega$ is pseudoconvex.
\medskip

{\parindent=0pt Then the pairs of assertions (i) and (ii) are equivalent, and if moreover $\wt\Omega$ has Lipschitz boundary, (i') and (ii'), are equivalent.}
\end{cor}

\section{Hearing pseudoconvexity with  $L^2$ Dolbeault cohomology}

 It is well known (see e.g. Corollary 4.2.6 and Theorem 4.2.9 in \cite{Hormander90}) that a domain $D$ in $\cb^n$ is pseudoconvex if and only if we have $H^{0,q}(D)=0$ for all $1\leq q\leq n-1$. This result also holds for the $L^2$-cohomology, provided $D$ satisfies interior($\ol D$)$=D$ (see \cite{Fu05} and references therein).
 We will  extend the results to
  the case when the  Dolbeault cohomology groups   with  $W^{s, p}$-forms are finite dimensional  for any given $s\ge 0$ and $p\ge 1$. (Here $W^{s, p}(D)$ is the $L^p$-Sobolev space of order $s$.)

\begin{thm}\label{pcnL2}
Let $D\subset\subset\cb^n$ be a bounded domain such that interior($\ol D$)$=D$. Let $s\ge 0$ and $p\ge 1$. If $H^{0,q}_{W^{s, p}}(D)$ is finite dimensional for all $1\le q\le n-1$, then $D$ is pseudoconvex.
\end{thm}

We will present a proof using an idea of Laufer \cite{Laufer66} as in the proof of Theorem~\ref{pcn}. The subtle difference is that while it makes sense to restrict a smooth form on a domain to the intersection of the domain with a complex hyperplane, restriction of an $L^p$-form to a complex hyperplane is not well-defined. This difficulty was overcome by appropriately modifying the construction of Laufer so that the factor $z_l$ in \eqref{dbar} is replaced by $z_l^m$ for a positive integer $m$. By choosing $m$ sufficiently large, we are able to make this restriction work. We now provide the detail, following \cite{Fu05}.  The following simple lemma illustrates the idea behind the construction of the forms $u_{\alpha, m}(k_1, k_2, \ldots, k_q)$ given by \eqref{eq:um} below.

\begin{lem}\label{lm:rest2} Let $v_1, \ldots, v_{n-1}\in L^p(D)$, $p\ge 1$, and
	let $m$ be a positive integer. Assume that $G$ is a continuous function on $D$ such that
	\begin{equation}\label{eq:rest2}
	G(z)=\sum_{j=1}^{n-1} z_j^m v_j(z).
	\end{equation}
If $m\ge 2(n-1)/p$,	then $G(0,\ldots, 0, z_n)=0$ for all $(0, \ldots, 0, z_n)\in D$.
\end{lem}

\begin{proof} Let $(0, \ldots, 0, z_n^0)\in D$. Write $z'=(z_1, \ldots, z_{n-1})$.
	Then for a sufficiently small positive numbers $a_1$ and $a_2$, we have
	\[
	D(a_1, a_2):=\{|z'|<a_1\}\times \{|z_n-z_n^0|<a_2\}\subset D.
	\]
	For any $\delta\in (0,\ 1)$, we have
	\begin{align*}
	\left(\int_{D(a_1, a_2)} |G(\delta z', z_n)|^p \, dV\right)^{1/p}
	&\le a^m_1 \delta^m \sum_{j=1}^{n-1}\left(\int_{D(a_1, a_2)} |v_j(\delta z', z_n)|^p\, dV\right)^{1/p}\\
	&\le a^m_1\delta^{m-2(n-1)/p}\sum_{j=1}^{n-1}\left(\int_{D(a_1\delta, a_2)} |v_j(z', z_n)|^p\, dV\right)^{1/p}\\
	&\le a^m\delta^{m-2(n-1)/p}\sum_{j=1}^{n-1}\left(\int_D |v_j(z)|^p\chi_{D(a_1\delta, a_2)}(z)\, dV\right)^{1/p}.
	\end{align*}
	Since $m\ge 2(n-1)/p$, letting $\delta\to 0$, we obtain from the Lebesgue dominated convergence
	theorem that
	\[
	\int_{D(a_1, a_2)} |G(0', z_n)|^p\, dV=0.
	\]
	Thus $G(0', z_n)=0$ for $|z_n-z_n^0|<a_2$. \end{proof}

\begin{proof}[Proof of Theorem~\ref{pcnL2}] The proof for $W^{s, p}$-cohomology is the same as for $L^p$-cohomology. For economy of notation, we will only provide the proof for $L^p$-cohomology. Proving by contradiction, we assume that $D$ is not pseudoconvex.
Then there exists a domain $\widetilde D\supsetneqq D$ such that every holomorphic
function on $D$ extends to $\widetilde D$.  After a
translation and a unitary transformation, we may assume that the
origin is in $\widetilde{D}\setminus\overline{D}$ and there is
a point $z^0$ in the intersection of
$z_n$-plane with $D$ that is in the same connected component
of $\widetilde{D}\cap\{z_1=0\}$ as the origin.

For any integers $\alpha \ge 0$, $m\ge 1$, $q\ge 1$,  $\{k_1, \ldots, k_{q-1}\}\subset \{1, 2, \ldots, n-1\}$ and $k_q=n$, let
\begin{equation}\label{eq:um}
u_{\alpha, m}(k_1, \ldots, k_q)=\frac{(\alpha+q-1)! \bar
	z_n^{m\alpha}(\bar z_{k_1}\cdots\bar z_{k_q})^{m-1}}{ r_m^{\alpha+
		q}}\sum_{j=1}^q (-1)^j \bar{z}_{k_j} \widetilde{d\bar z_{k_j}}
\end{equation}
where $r_m=|z_1|^{2m}+\ldots +|z_n|^{2m}$. Evidently,
$u(k_1, \ldots, k_q)$ is a smooth form on
$\C^n\setminus\{ 0\}$. Since $0\notin \overline{D}$, $u(k_1, \ldots, k_q)\in L^p_{(0, q-1)}(D)$.
Moreover, $u(k_1, \ldots, k_q)$ is skew-symmetric with respect to the
indices $(k_1, \ldots, k_{q-1})$.  In particular, $u(k_1, \ldots, k_q)=0$ when two $k_j$'s are identical.

 For $K=(k_1,\ldots, k_q)$,  write $d\bar z_K=d\bar z_{k_1}\wedge \ldots\wedge d\bar z_{k_q}$, $\bar z_K^{m-1}= (\bar z_{k_1}\cdots\bar
 z_{k_q})^{m-1}$.  Denoted by
$(k_1, \ldots, k_q \setminus J)$ the tuple of remaining indices after deleting those
in $J$ from $(k_1, \ldots, k_q)$. It follows from straightforward computations that

\begin{align*}\notag
	\dbar u_{\alpha, m}(k_1, \ldots, k_q)&=-\frac{(\alpha+q)!m \bar
		z_n^{m\alpha}\bar z_K^{m-1}} {r_m^{\alpha+q+1}}\big(r_m d\bar z_K \\
	&\quad +\big(\sum_{\ell =1}^n \bar z_\ell^{m-1} z_\ell^m d\bar
	z_\ell\big)\wedge \big(\sum_{j=1}^q (-1)^j \bar
	z_{k_j}\widetilde{d\bar z_{k_j}}\big)\big)\\
	&=m\sum_{\ell=1}^{n-1} z^m_\ell u_{\alpha, m}(\ell, k_1, \ldots,
	k_q).
\end{align*}

In particular, $u_{\alpha, m}(1, \ldots, n)$ is $\dbar$-closed. Our
next goal is to solve the $\dbar$-equation in $L^p$-spaces
inductively with the $(0, n-1)$-forms $u_{\alpha, m}(1, \ldots, n)$
as the initial data, and eventually produce an $L^p$-holomorphic
function on $D$. This holomorphic function has a holomorphic
extension to $\widetilde D$. By way of the construction, the
extension has singularity at the origin, which leads to a
contradiction. We now provide the details.

We fix $m\ge 2(n-1)/p$. Let $M$ be an integer such that $M>\dim H^{0, q}_{L^p}(D)$ for
all $1\le q\le n-1$. Let $\scriptf_0$ be the linear
span of $\{u_{\alpha, m}(1, \ldots, n); \ \alpha=1, \ldots,
M^{n-1}\}$. For any $u\in\scriptf_0$ and for any $\{k_1, \ldots,
k_{q-1}\}\subset\{1, \ldots, n-1\}$, we set
\[
u(k_1, \ldots, k_{q-1}, n)=\sum_{j=1}^k c_j u_{\alpha_j, m}(k_1,
\ldots, k_{q-1}, n)
\]
if $u=\sum_{j=1}^k c_j u_{\alpha_j, m}(1, \ldots, n)$.  We decompose
$\scriptf_0$ into a direct sum of $M^{n-2}$ subspaces, each of which
is $M$-dimensional. Since $\dim H^{0, n-1}_{L^p}(D)<M$ and $u_{\alpha, m}(1,
\ldots, n)\in \kernel(\dbar_{n-1})$, there exists a non-zero form
$u$ in each of the subspaces such that $\dbar v_u(\emptyset)=u$ for
some $v_u(\emptyset)\in L^p_{(0, n-2)}(D)$.  Let $\scriptf_1$
be the $M^{n-2}$-dimensional linear span of all such $u$'s.  We
extend $u\mapsto v_u(\emptyset)$ linearly to all $u\in\scriptf_1$.

For $0\le q\le n-1$, we use induction on $q$ to construct an
$M^{n-q-2}$-dimensional subspace $\scriptf_{q+1}$ of $\scriptf_q$
with the properties that for any $u\in\scriptf_{q+1}$, there exists
$v_u(k_1, \ldots, k_q)\in L^p_{(0, n-q-2)}(D)$ for all $\{k_1,
\ldots, k_q\}\subset\{1, \ldots, n-1\}$ such that $v_u(k_1, \ldots, k_q)$ depends linearly on $u$;  $v_u(k_1, \ldots, k_q)$ is skew-symmetric with respect to
indices $K=(k_1, \ldots, k_q)$; and
\[
\dbar v_u(K)=m\sum_{j=1}^q (-1)^j z^m_{k_j} v_u(K \setminus
k_j) + (-1)^{q+|K|} u(1, \ldots, n\setminus K),
\]
where $|K|=k_1+\ldots +k_q$.

We now show how to construct $\scriptf_{q+1}$ and $v_u(k_1, \ldots,
k_q)$ for $u\in\scriptf_{q+1}$ and $\{k_1,\ldots, k_q\}\subset\{1,
\ldots, n-1\}$ once $\scriptf_q$ has been constructed.  For any
$u\in\scriptf_q$ and any $\{k_1, \ldots, k_q\} \subset\{1, \ldots,
n-1\}$, write $K=(k_1, \ldots, k_q)$, and let
\[
w_u(K)=m\sum_{j=1}^q (-1)^j z^m_{k_j} v_u(K\setminus k_j)
+(-1)^{q+|K|} u(1, \ldots, n\setminus K) .
\]
Then as in the previous case,
\[
\dbar w_u(K)
=(-1)^{q+|K|}\big( -m\sum_{j=1}^q z^m_{k_j} u(k_j, (1, \ldots, n \setminus
K)) + \dbar u(1, \ldots, n\setminus K)\big)=0.
\]
We again decompose $\scriptf_q$ into a direct sum of $M^{n-q-2}$
linear subspaces, each of which is $M$-dimensional.  Since
$\dim(H^{0,n-q-2}_{L^p}(D))<M$ and $\dbar w_u(K)=0$, there exists a non-zero
form $u$ in each of these subspaces such that $\dbar v_u(K)=w_u(K)$
for some $v_u(K)\in L^p_{(0, n-q-2)}(D)$. Since $w_u(K)$ is
skew-symmetric with respect to indices $K$, we may choose $v_u(K)$
to be skew-symmetric with respect to $K$ as well. The subspace
$\scriptf_{q+1}$ of $\scriptf_q$ is then the linear span of all such
$u$'s.

Note that $\dim(\scriptf_{n-1})=1$.  Let $u$ be any non-zero form in
$\scriptf_{n-1}$ and let
\[
F(z)=w_u(1, \ldots, n-1)= m\sum_{j=1}^{n-1} z^m_j v_u(1, \ldots, n-1\setminus j) - (-1)^{n+\frac{n(n-1)}2} u(n) .
\]
Then $F\in L^p(D)$ and $\dbar F =0$.  Therefore, $F$ is
holomorphic on $D$ and hence has a holomorphic extension to
$\widetilde{D}$.  Restricting to $z_n$-plane, by Lemma~\ref{lm:rest2}, we have
\[
F(0', z_n)=- (-1)^{n+\frac{n(n-1)}2} u(n) (0', z_n)=(-1)^{n+\frac{n(n-1)}2}\sum_{\alpha=1}^M c_\alpha\frac{\alpha!}{z_n^{m(\alpha+1)}}
\]
where the $c_\alpha$'s are constants, not all zeros. This contradicts the analyticity of $F$ near the origin. We therefore conclude the proof of Theorem~\ref{pcnL2}.
\end{proof}

\begin{cor}\label{eqW1}
Let $D\subset\subset\wt\Omega$ be two relatively compact open subsets of $\cb^n$, $n\geq 2$, such that both $\cb^n\setminus\wt\Omega$ and $\wt\Omega\setminus D$ are connected. Assume $D$ has Lipschitz boundary. Then $H^{0,q}_{W^1_{loc}}(\wt\Omega\setminus D)=0$, for all $1\leq q\leq n-2$, and $H^{0,n-1}_{W^1_{loc}}(\wt\Omega\setminus D)$ is Hausdorff if and only if $\wt\Omega$ and $D$ are pseudoconvex.
\end{cor}
\begin{proof}
The necessary condition is a direct consequence of Corollary \ref{caractL2} and  Theorem \ref{pcnL2}. The sufficient condition follows from H\"ormander vanishing $L^2$-theory  and  Corollary \ref{caractL2}.
\end{proof}

  Next we give a characterisation of the annulus domain by its $L^2$ Dolbeault  cohomology when the inner hole  $D$ has $\cc^2$ boundary.

\begin{cor}\label{eqL2}
Let $D\subset\subset\wt\Omega$ be two relatively compact open subsets of $\cb^n$, $n\geq 2$, such that both $\cb^n\setminus\wt\Omega$ and $\wt\Omega\setminus D$ are connected. Assume  $D$ has $\cc^2$ boundary and $\wt\Omega$ has Lipschitz boundary.
Then  $H^{0,q}_{L^2}(\wt\Omega\setminus D)=0$, for all $1\leq q\leq n-2$, and $H^{0,n-1}_{L^2}(\wt\Omega\setminus D)$ is Hausdorff if and only if $\wt\Omega$ and $D$ are pseudoconvex.
\end{cor}
\begin{proof}
We first  prove  sufficiency.  If   $D$ has  $\cc^3$ boundary and $n\geq 3$, this follows directly  from
\cite{Shaw85}.   If  the boundary is only $\cc^2$, it follows from Theorem 3 in \cite{Har} that  $H^{n,q}_{W^1}(D)=0$, for all $1\leq q\leq n-1$. (When the boundary is $\cc^\infty$, this follows from the work of Kohn~\cite{Ko74}.)  The sufficiency then follows from Corollary \ref{caractL2}.

The necessary condition is a direct consequence of Corollary \ref{caractL2} and  Theorem \ref{pcnL2}.
\end{proof}

 Next we will set up the spectral theory for the $\dbar$-Neumann operator. Let $X$ be a complex manifold of dimension $n$ equipped with a  hermitian metric.
 Let $\Omega$ be a bounded domain in $X$.
Let
\[
Q^{\Omega}_{p, q}(u, v)=\langle\dbar_{p,q} u, \dbar_{p,q}
v\rangle_\Omega+\langle\dbarstar_{p,q-1} u,
\dbarstar_{p,q-1} v\rangle_\Omega
\]
be the sesquilinear form on $L^2_{p, q}(\Omega)$ with domain of definition
$\dom(Q^{\Omega}_{p, q})=\dom(\dbar_{p, q})\cap \dom(\dbarstar_{p, q-1})$.
Then $Q^{\Omega}_{p, q}$ is densely defined and closed. It then
follows from general operator theory (see \cite{Davies95}) that $Q^\Omega_{p, q}$
uniquely determines a densely defined, non-negative, self-adjoint operator
$\square^{\Omega}_{p, q}\colon L^2_{p, q}(\Omega)\to L^2_{p,
	q}(\Omega)$ such that $\dom((\square^{\Omega}_{p, q})^{1/2})=\dom(Q^{\Omega}_{p, q})$
and
\[
Q^{\Omega}_{p, q}(u, v)=\langle (\square^{\Omega}_{p, q})^{1/2} u, (\square^{\Omega}_{p, q})^{1/2} v\rangle, \qquad \text{for}\ u, v\in \dom(Q^{\Omega}_{p, q}).
\]
Moreover,
\[
\dom(\square^{\Omega}_{p, q})=\{ u\in \dom(Q^{\Omega}_{p, q})\mid \dbar_{p,q} u\in \dom(\dbarstar_{p, q}), \dbarstar_{p, q-1} u\in \dom(\dbar_{p, q-1})\}.
\]
The operator $\square^{\Omega}_{p, q}$ is {\it the
	$\dbar$-Neumann Laplacian} on $L^2_{p, q}(\Omega)$. (We refer the reader to \cite[\S2]{Fu10} for
a spectral theoretic setup for the $\dbar$-Neumann Laplacian.) We will drop the superscript and/or subscript from $\square^\Omega_{p, q}$ when
their appearances are either inconsequential or clear from the context.

Let $\sigma(\square_{p, q})$ be the spectrum of $\square_{p, q}$. Recall that $\sigma(\square_{p, q})$ is the complement in $\C$ of the resolvent set which consists of all $\lambda\in\C$ such that $\lambda I-\square_{p, q}\colon \dom(\square_{p, q})\to L^2_{p, q}(\Omega)$ is one-to-one, onto, and has bounded inverse. (See \cite{Davies95} for relevant material on spectral theory of differential operators.) Since $\square_{p,q }$ is a non-negative self-adjoint operator on a Hilbert space,  $\sigma(\square_{p, q})$ is a non-empty closed subset of the interval $[0, \ \infty)$. Let $\sigma_e(\square_{p, q})$ be the essential spectrum of $\square_{p, q}$; namely, points in $\sigma(\square_{p, q})$ that  are either isolated points of the spectrum but eigenvalues of infinity multiplicity; or limit points of the spectrum. By definition, the essential spectrum $\sigma_e(\square_{p, q})$ is also a closed subset and the set of limit points of $\sigma_e(\square_{p, q})$ is the same as that of $\sigma(\square_{p, q})$.  We summarize the following spectral theoretic interpretations for positivity of the $\dbar$-Neumann Laplacian $\square_{p, q}$ in the following proposition:

\begin{prop}\label{prop:spectral} Let $\square_{p, q}$ be the $\dbar$-Neumann Laplacian on $(p, q)$-forms.
	
	\begin{enumerate}
		\item $0$ is not a limit point of $\sigma (\square_{p, q})$ if and only if both $\range(\dbar_{p, q-1})$ and $\range(\dbar_{p, q})$ are closed.
		\item $0\notin\sigma_e(\square_{p, q})$ if and only if $\range(\dbar_{p, q-1})$ and $\range(\dbar_{p, q})$ are closed, and $H^{p, q}_{L^2}(\Omega)$ is finite dimensional.
		
		\item $0\notin\sigma(\square_{p, q})$ if and only if $\range(\dbar_{p, q-1})$ and $\range(\dbar_{p, q})$ are closed, and $H^{p, q}_{L^2}(\Omega)$ is trivial.
	\end{enumerate}
\end{prop}

We refer the reader to \cite[\S1.1]{Hormander65} (see also \cite[Appendix A]{Hormander04}) for proofs of (1) and (3) and to \cite[\S2]{Fu05}(see also \cite[\S2]{Fu10}) for a proof of (2).

  Recall that, for a bounded domain $\Omega$ in a complex, hermitian, $n$-dimensional manifold, the top degree $L^2$-cohomology groups $H^{p,n}_{L^2}(\Omega)$, $0\leq p\leq n$, always vanish, hence $\range(\dbar_{p, n-1})$ is closed in this case. So, in top degree, the first assertion of Proposition~\ref{prop:spectral} becomes:
$0$ is not a limit point of $\sigma (\square_{p, n-1})$ if and only if  $\range(\dbar_{p, n-2})$ is closed.  Combining Proposition~\ref{prop:spectral} with Theorems~\ref{dualL2} and \ref{pcnL2}, we then have:


\begin{thm}\label{th:positivity1}  Let $X$ be a Stein manifold of dimension $n\ge 3$, equipped with a hermitian metric.  Let $\Omega=\widetilde{\Omega}\setminus \overline{D}$ where $\widetilde{\Omega}$ is a relatively compact domain with  connected complement in $X$ and $D\subset\subset\wt{\Omega}$ is an open set with connected complement in $\widetilde{\Omega}$ and with $C^2$-boundary. If both $\widetilde{\Omega}$ and  $D$ are pseudoconvex,
 then there exists a constant $C>0$ such that
	\begin{equation}\label{eq:inf1}
	\inf{\sigma(\square^\Omega_{p, q})} \ge C
	\end{equation}
	for all $0\le p\le n$ and $1\le q\le n-2$ and
	\begin{equation}\sigma(\square^\Omega_{p, n-1})\cap (0, \ C)=\emptyset.\end{equation}
\end{thm}

The converse of the above theorem also holds. We summarize the results  in a slight more general form as follows:

\begin{thm}\label{th:chaS} Let $\widetilde{\Omega}$ be a bounded domain with connected complement in  a hermitian  Stein manifold and let $D$ be a relatively compact open subset of $\widetilde{\Omega}$ with connected complement. Let $\Omega=\widetilde{\Omega}\setminus \overline{D}$. Suppose  $\widetilde{\Omega}$ and $D$ have Lipschitz boundary.  Fix $0\le p\le n$. If $0\not\in\sigma_e(\square_{p, q})$ for $1\le q\le n-2$ when $n\ge 3$ or if $0$ is not a limit point for $\sigma_e(\square_{p, 1})$ when $n=2$, then both $\widetilde{\Omega}$ and $D$ are pseudoconvex.
\end{thm}

The above theorem is a consequence of Proposition~\ref{prop:spectral} and the characterization  of pseudoconvexity by $L^2$-cohomology groups in Section~\ref{sec:chaL2}. Note that when $n\ge 3$, one only need to assume the positivity of $\sigma_e(\square_{p, q})$ for $1\le q\le n-2$. For $n=2$, we use, as   noted above, that $\range(\dbar_{p, n-1})$ is always closed.  For completeness, we provide the proof of Theorem~\ref{th:chaS}. We first establish the following spectral theoretic version of the Hartogs phenomenon, as in Theorem~\ref{L2}. Without loss of generality, we will deal only with $(0, q)$-forms.

\begin{lem}\label{lm:hartogs} Let $\widetilde{\Omega}$ be a domain in a Stein manifold $X$ equipped with a hermitian metric and let $K$ be a compact subset of $\widetilde{\Omega}$.  Let $\Omega=\widetilde{\Omega}\setminus K$. Then
	\begin{enumerate}
		\item $\inf\sigma_e(\square_q^{\widetilde\Omega})>0$ provided $\inf\sigma_e(\square_q^{\Omega})>0$.
		\item $\inf\sigma(\square_q^{\widetilde\Omega})>0$ provided $\inf\sigma(\square_q^{\Omega})>0$
	\end{enumerate}	
\end{lem}

\begin{proof}  To prove (1), from Lemma \ref{lm:hartogs1} and Proposition~\ref{prop:spectral}, it suffices to show that $H^{0, q}_{L^2}(\wt{\Omega})$ is finite dimensional provided $H^{0, q}_{L^2}(\Omega)$ is finite
	dimensional.  Let $R\colon \kernel(\dbar^{\wt{\Omega}}_{q}) \to \kernel(\dbar^{\Omega}_{q})$ be the restriction map $\beta\mapsto \beta\vert_\Omega$.  Repeating  arguments
	used in Lemma \ref{lm:hartogs1} (with $\dbar f$ replaced by $\beta$) yields that $R$ induces an injective homomorphism from $H^{0, q}_{L^2}(\wt{\Omega})$ into $H^{0, q}_{L^2}(\Omega)$. Therefore,    $\dim H^{0, q}_{L^2}(\wt{\Omega})\le \dim H^{0, q}_{L^2}(\Omega)$. This concludes the proof of $(1)$ and hence that of $(2)$ . \end{proof}

The following lemma is a spectral theoretic interpretation of Theorem~\ref{dualL2} (in a slightly more general form):

\begin{lem}\label{lm:hartogs4} Let $D$ be a relatively compact open set in a Stein manifold $X$ of dimension $n\ge 2$ with connected complement and Lipschitz boundary. Let $\Omega=X\setminus \overline{D}$.  Then $0\not\in\sigma_e(\square^\Omega_{q})$ for $1\le q\le n-2$ and $0$ is not a limit point for $\sigma_e(\square^\Omega_{n-1})$ if and only if $\dim H^{n, q}_{W^1}(D)<\infty$, $1\le q\le n-2$, and $\dim H^{n, n-1}_{W^1}(D)=0$.	
\end{lem}

\begin{proof} Note that by Proposition~\ref{prop:spectral},  as already mentioned above, $0$ is not  a limit point of $\sigma_{e}(\square^\Omega_{n-1})$ is equivalent to $\range(\dbar^\Omega_{n-2})$ is closed, which is equivalent to $H^{0, n-1}_{L^2}(\Omega)$ is Hausdorff. By
	\cite[Proposition~4.7]{LaShdualiteL2}, this is also equivalent to $H^{n, n-1}_{W^1}(D)=\{0\}$. In light of Proposition~\ref{prop:spectral} and by $L^2$-Serre duality, $\dim H^{0, q}_{L^2}(\Omega)=\dim H^{n, n-q}_{\overline{\Omega}, L^2}(X)$ for $1\le q\le n-1$. So it remains to show that for $1\le q\le n-2$, $\dim H^{n, q}_{W^1}(D)=\dim H^{n, q+1}_{\overline{\Omega}, L^2}(X)$, provided one of these quantities is finite.
	
	Suppose $\dim H^{n, q+1}_{\overline{\Omega}, L^2}(X)=N<\infty$.  Let $g_j$, $1\le j\le N$, be $\dbar$-closed $(n, q+1)$-forms supported on $\overline{\Omega}$ such that $\{[g_j]\}_{j=1}^N$ spans $H^{n, q+1}_{\overline{\Omega}, L^2}(X)$.  Since $X$ is Stein, there exists $(n, q)$-forms $h_j$ with $L^2_{\rm loc}$-coefficients such that $\dbar h_j=g_j$. (Here, we identify $g_j$ with its extension to $X$ by setting $g_j=0$ outside $\overline{D}$.) Since $g_j$ is supported on $X\setminus D$, $\dbar h_j=0$ on $D$.  Now let $f\in W^1_{n, q}(D)$ be any $\dbar$-closed form.  Since $D$ has Lipschitz boundary, there exists an extension $\tilde f\in W^{1}_{n, q}(X)$ of $f$ to $X$. Since $\dbar\tilde f=0$ on $D$, under the assumption, there exists a $g\in L^2_{n, q}(X)$, supported on $X\setminus D$, such that
	\[
	\dbar \tilde{f}=\dbar g+\sum_{j=1}^N c_j g_j=\dbar g+\sum_{j=1}^N c_j\dbar h_j,
    \]
   for some constants $c_j\in\C$, $1\le j\le N$.  Therefore, there exists $(n, q-1)$-form $u$ with $L^2_{\rm loc}$-coefficients such that
    \[
    \tilde f=g+\sum_{j=1}^N c_j h_j+\dbar u.
    \]
Note that using the interior ellipticity of $\dbar\oplus \dbarstar$, we may choose the forms $h_j$ and $u$ to have $W^1$-coefficients on $D$. Restricting to $D$, we then have $f=\sum_{j=1}^N c_j h_j +\dbar u$, which implies that $\dim H^{n, q}_{W^1}(D)\le N$.

Conversely, suppose $\dim H^{n, q}_{W^1}(D)= N<\infty$.  Let $\{g_j\}_{j=1}^N\subset W^1_{n, q}(D)$ be $\dbar$-closed forms such that $\{[g_j]\}_{j=1}^N$ spans $H^{n, q}_{W^1}(D)$.   Let $f\in L^2_{n, q+1}(X)$ be a $\dbar$-closed form with compact support in $X\setminus D$.  Since $X$ is a Stein manifold, there exists an $(n, q)$-form $u$ with $L^2_{\rm loc}$-coefficients such that $\dbar u=f$. Using the interior ellipticity of $\dbar\oplus\dbarstar$, we may choose $u$ so that $u\in W^1_{n, q}(D)$. Since $f$ is supported on $X\setminus D$, $\dbar u=0$ on $D$. Under the assumption,
there exists $h\in W^1_{n, q-1}(D)$ such that
\[
u=\dbar h+ \sum_{j=1}^N b_j g_j,
\]
for some constants $b_j\in \C$, $1\le j\le N$.  Let $\tilde{g_j}, \tilde{h}\in W^1_{n, q}(X)$ be any extensions of $g_j$ and $h$ respectively from $D$ to $X$ with compact supports in $X$. Let $g=u-\dbar \tilde h-\sum_{j=1}^N b_j \tilde g_j$.  Then $g$ is compactly supported on $X\setminus D$ and
\[
\dbar g=f-\sum_{j=1}^N b_j \dbar\tilde g_j,
\]
which implies that $\dim H^{n, q+1}_{\overline{\Omega}, L^2}(X)\le N$.  \end{proof}

We are now in a position to prove Theorem~\ref{th:chaS}. By Lemma~\ref{lm:hartogs}, Lemma~\ref{lm:hartogs1} and Theorem~\ref{L2} (ii), we know that $H^{0, q}_{L^2}(\wt{\Omega})$, $1\le q\le n-2$, are finite dimensional
and $H^{0, n-1}_{L^2}(\wt{\Omega})$ is trivial. By Theorem~\ref{pcnL2}, $\wt{\Omega}$ is pseudoconvex.

Applying Lemma~\ref{lm:hartogs4} with $X=\wt{\Omega}$, we then have $\dim H^{0, q}_{W^1}(D)<\infty$, $1\le q\le n-2$, and $H^{0, n-1}_{W^1}(D)=\{0\}$.  Applying Theorem~\ref{pcnL2} again, we then conclude that $D$ is pseudoconvex.

\begin{rem}  In Corollary~\ref{eqW1},  we only need to assume that   the boundary is   Lipschitz  when we use Dolbeault cohomology with $W^1$ coefficients.  But in Corollary~\ref{eqL2}, the boundary $D$ needs to be $\cc^2$ smooth.
Note that the necessary condition in Corollary~\ref{eqL2} still holds if the boundary of $D$ is only Lipschitz.
We do not know if we can replace the $C^2$ assumption by the  Lipschitz condition  in Corollary~\ref{eqL2} or Theorem~\ref{th:positivity1}.
We conjecture that  they still hold   if the boundary of $D$ is only Lipschitz. This has been verified     when the inner domain $D$ is a product domain or piecewise smooth pseudoconvex domain  (see the   results in  \cite{ChLaSh}).  In particular, when the domain $\Omega$  is the annulus between a ball and a bidisc in $\mathbb C^2$, one has that
$H^{0,1}_{L^2}(\Omega)$ is Hausdorff.  This yields the $W^1$ estimates for $\dbar$ on bidisc using  Corollary~\ref{caractL2}.  The general  case with Lipschitz holes   is  still an  open problem.
\end{rem}

\begin{rem} Suppose that the number  of holes in $\wt\Omega\setminus \Omega= \ol D$ is  infinite and each component is pseudoconvex,  the boundary $\Omega$ is not Lipschitz. But  one still  can have   $interior \overline \Omega=\Omega$. We do not know if the $\dbar$-Neumann operator $\square_{p,q}$ has closed range
for $1\le q\le n-1$.  In fact, one does not even know if the classical  Neumann operator has closed range.
\end{rem}

\enddocument

\end